\documentclass[12pt,reqno]{amsart}
\usepackage[headings]{fullpage}
\usepackage{amsmath,amssymb,amsthm}
\usepackage{bbold}
\usepackage[bookmarks=true,%
colorlinks=true,%
linkcolor=blue,%
citecolor=blue,%
filecolor=blue,%
menucolor=blue,%
urlcolor=blue,%
breaklinks=true]{hyperref}
\usepackage{graphicx,subcaption,url}
\usepackage{fullpage,comment}
\usepackage[shortlabels]{enumitem}
\setlist{font=\normalfont}
\usepackage{mathrsfs,mathtools}
\usepackage{tikz,tikz-cd}
\usetikzlibrary{knots,hobby,math,decorations.markings}
\usepackage{ifthen}
\allowdisplaybreaks

\makeatletter
\renewcommand*{\fps@figure}{htpb!}
\makeatother

\newtheorem{theorem}{Theorem}[section]
\theoremstyle{definition}
\newtheorem{proposition}[theorem]{Proposition}
\newtheorem{lemma}[theorem]{Lemma}
\newtheorem{definition}[theorem]{Definition}
\newtheorem{remark}[theorem]{Remark}
\newtheorem{corollary}[theorem]{Corollary}

\DeclareMathAlphabet{\AMSbb}{U}{msb}{m}{n}

\providecommand{\abs}[1]{\left|{#1}\right|}

\DeclareMathOperator{\qtr}{\widehat{tr}}

\DeclareMathOperator{\Span}{span}

\newcommand{\term}[1]{\textit{#1}}

\newcommand{\cx}{\AMSbb{C}}
\newcommand{\nats}{\AMSbb{N}}
\newcommand{\ints}{\AMSbb{Z}}
\newcommand{\reals}{\AMSbb{R}}

\newcommand{\iunit}{\mathbf{i}}

\newcommand{\ground}{\ints[\iunit][q^{\pm1/8}]}
\newcommand{\groundev}{\ints[q^{\pm1/2}]}
\newcommand{\indexR}{\ints(\!(q^{1/2})\!)}
\newcommand{\indexRp}{\ints[\![q^{1/2}]\!]}

\newcommand{\SL}{\mathrm{SL}}
\newcommand{\PSL}{\mathrm{PSL}}

\newcommand{\skein}{\mathcal{S}}
\newcommand{\skeinev}{\skein^{\mathrm{ev}}}

\newcommand{\skeincr}{\skein_{\mathrm{cr}}}
\newcommand{\skeincrev}{\skein_{\mathrm{cr}}^{\mathrm{ev}}}

\newcommand{\surface}{\Sigma}
\newcommand{\surclose}{\overline{\Sigma}}
\newcommand{\marked}{\mathcal{P}}
\newcommand{\heeg}{\mathcal{H}}

\newcommand{\cut}{\Theta}

\newcommand{\annulus}{\AMSbb{A}}
\newcommand{\lantern}{\AMSbb{L}}

\newcommand{\qtorus}{\AMSbb{T}}
\newcommand{\qtorusev}{\qtorus^{\mathrm{ev}}}
\newcommand{\qglue}{\widehat{\mathcal{G}}}
\newcommand{\qgluev}{\widehat{\mathcal{G}}^{\mathrm{ev}}}
\newcommand{\qvar}[1]{\hat{#1}}
\newcommand{\qz}{\qvar{z}}
\newcommand{\qw}{\qvar{w}}
\newcommand{\qx}{\qvar{x}}
\newcommand{\qy}{\qvar{y}}
\newcommand{\qv}{\qvar{v}}
\newcommand{\qm}{\qvar{m}}

\newcommand{\qZ}{\qvar{Z}}
\newcommand{\qW}{\qvar{W}}
\newcommand{\qX}{\qvar{X}}
\newcommand{\qY}{\qvar{Y}}
\newcommand{\qV}{\qvar{V}}

\newcommand{\qe}{\qvar{e}}
\newcommand{\qE}{\qvar{E}}

\newcommand{\ideal}[1]{\langle#1\rangle}
\newcommand{\lideal}[1]{\vphantom{\rangle}^{}_{L\!}\ideal{#1}}
\newcommand{\quotbyL}[1]{/\!\lideal{#1}}
\newcommand{\rideal}[1]{\ideal{#1}^{}_{\!R}}

\newcommand{\triang}{\mathcal{T}}

\newcommand{\Emon}{\mathcal{E}}

\newcommand{\embed}{\hookrightarrow}
\newcommand{\onto}{\twoheadrightarrow}

\tikzset{knot diagram/every knot diagram/.style={
    background color=gray!20,clip width=5,end tolerance=5pt,clip radius=0.1cm}}
\newcommand{\stsize}{\scriptsize}
\newenvironment{linkdiag}[1][0.5]{\,%\mathop{}\!
  \begin{tikzpicture}[scale=0.8,baseline=(ref.base)]
    \node (ref) at (0.5,{#1}){\phantom{$-$}};}{\end{tikzpicture}\,%\!\mathop{}
	}
\tikzset{-o-/.code 2 args={
    \pgfkeysalso{decoration={markings,mark=at position #1 with {\arrow{#2}}},
      postaction={decorate}}
}}
\newcommand{\stnode}[1]{node[inner sep=1pt,right]{\stsize$#1$}}
\newcommand{\stnodel}[1]{node[inner sep=1pt,left]{\stsize$#1$}}

\newcommand{\relempty}[1][]{
\begin{linkdiag}
\fill[gray!20] (0,0)rectangle(1,1);
\draw[#1] (1,0)--(1,1);
\end{linkdiag}
}
\newcommand{\relcross}[2]{
\begin{linkdiag}
\fill[gray!20] (0,0)rectangle(1,1);
\draw[->] (1,0)--(1,1);
\begin{knot}
\strand[thick] (0,0.3)..controls +(0.5,0) and +(-0.5,0)..(1,0.7);
\strand[thick] (0,0.7)..controls +(0.5,0) and +(-0.5,0)..(1,0.3);
\end{knot}
\draw (1,0.7)\stnode{#1} (1,0.3)\stnode{#2};
\end{linkdiag}
}
\newcommand{\relarc}[3][->]{
\begin{linkdiag}
\fill[gray!20] (0,0)rectangle(1,1);\draw[#1] (1,0)--(1,1);
\draw[thick] (1,0.7)..controls(0.2,0.7) and (0.2,0.3)..(1,0.3);
\draw (1,0.7)\stnode{#2} (1,0.3)\stnode{#3};
\end{linkdiag}
}

\newcommand{\relacross}[3][->]{
\begin{linkdiag}
\fill[gray!20] (0,0)rectangle(1,1);\draw[#1] (1,0)--(1,1);
\draw[thick] (0,0.7)--(1,0.7) (0,0.3)--(1,0.3);
\draw (1,0.7)\stnode{#2} (1,0.3)\stnode{#3};
\end{linkdiag}
}
\newcommand{\relcorner}[3][]{
\begin{linkdiag}
\fill[gray!20] (0,-0.1)rectangle(1,1.1);\draw[#1] (1,-0.1)--(1,1.1);
\draw[thick] (1,0.8)..controls(0.2,0.8) and (0.2,0.2)..(1,0.2);
\draw[fill=white] (1,0.5)circle(0.07);
\draw (1,0.8)\stnode{#2} (1,0.2)\stnode{#3};
\end{linkdiag}
}
\newcommand{\relbottom}[1]{
\begin{linkdiag}
\fill[gray!20] (0,0)rectangle(1,1);\draw (1,0)--(1,1);
\draw[thick] (0,0.2)--(1,0.2);
\draw[fill=white] (1,0.5)circle(0.07);
\draw (1,0.2)\stnode{#1};
\end{linkdiag}
}
\newcommand{\reltwup}[1]{
\begin{linkdiag}
\fill[gray!20] (0,0)rectangle(1,1);\draw (1,0)--(1,1);
\draw[thick] (0,0.2)..controls (0.5,0.2) and (0.5,0.8)..(1,0.8);
\draw[fill=white] (1,0.5)circle(0.07);
\draw (1,0.8)\stnode{#1};
\end{linkdiag}
}

\newcommand{\relaround}[1]{
\begin{linkdiag}
\fill[gray!20] (-0.2,-0.2)rectangle(1.2,1.2);
\draw[fill=white] (0.5,0.5)circle(0.25);
\foreach \t in {0,120,240}
\draw[fill=white] (0.5,0.5) +(\t:0.25)circle(0.07);
\draw[thick] (-0.2,0.5)..controls +(0.35,0) and +(-0.35,0)..(0.5,#1)
..controls +(0.35,0) and +(-0.35,0).. (1.2,0.5);
\end{linkdiag}
}

\newcommand{\red}[1]{{\color{red} #1}}

\def\BN{\AMSbb N}

\def\BZ{\AMSbb Z}
\def\BQ{\AMSbb Q}
\def\BR{\AMSbb R}
\def\BC{\AMSbb C}

\def\calN{\mathcal N}

\def\s{\sigma}

\def\a{\alpha}

\def\d{\delta}

\def\be{\begin{equation}}
\def\ee{\end{equation}}

\def\Span{\mathrm{Span}}

\def\DGG{\mathrm{DGG}}

\def\BD{\mathbb{\Delta}}
\def\BE{\AMSbb E}
\def\sym{\mathrm{sym}}
\def\Zar{\vec{Z}}

\makeatletter

\renewcommand\thepart{\@Roman\c@part}%
\renewcommand\part{%
\if@noskipsec \leavevmode \fi
\par
\addvspace{6.7ex}%
\@afterindentfalse
\secdef\@part\@spart}
\def\@part[#1]#2{%
\ifnum \c@secnumdepth >\m@ne
\refstepcounter{part}%
\addcontentsline{toc}{part}{Part~\thepart.\ #1}%
\else
\addcontentsline{toc}{part}{#1}%
\fi
{\parindent \z@ \raggedright
\interlinepenalty \@M
\normalfont
\ifnum \c@secnumdepth >\m@ne
\centering\large\scshape \partname~\thepart.%
\hspace{1ex}%
\fi%
\large\scshape #2%
\markboth{}{}\par}%
\nobreak
\vskip 4.7ex
\@afterheading}
\def\@spart#1{
\refstepcounter{part}%
\addcontentsline{toc}{part}{#1}%
{\parindent \z@ \raggedright
\interlinepenalty \@M
\normalfont
\centering\large\scshape #1\par}%
\nobreak
\vskip 4.7ex
\@afterheading}
\renewcommand*\l@part[2]{%
\ifnum \c@tocdepth >-2\relax
\addpenalty\@secpenalty
\addvspace{0.75em \@plus\p@}%
\begingroup
\parindent \z@ \rightskip \@pnumwidth
\parfillskip -\@pnumwidth
{\leavevmode
\normalsize \bfseries #1\hfil \hb@xt@\@pnumwidth{\hss #2}}\par
\nobreak
\if@compatibility
\global\@nobreaktrue
\everypar{\global\@nobreakfalse\everypar{}}%
\fi
\endgroup
\fi}

\def\l@subsection{\@tocline{2}{0pt}{2pc}{6pc}{}}
\makeatother

\begin{document}
%\title{The 3d-quantum trace map and the 3d-index}
%\title{From the 3d-skein module to the 3d-index via the quantum trace map} 
\title{The 3d-index of the 3d-skein module via the quantum trace map} 

\author{Stavros Garoufalidis}
\address{% Max Planck Institute for Mathematics \\
% Vivatsgasse 7, 53111 Bonn, GERMANY \newline
Shenzhen International Center for Mathematics, Department of Mathematics \\
Southern University of Science and Technology \\
1088 Xueyuan Avenue, Shenzhen, Guangdong, China \newline
{\tt \url{http://people.mpim-bonn.mpg.de/stavros}}}
\email{stavros@mpim-bonn.mpg.de}

\author{Tao Yu}
\address{Shenzhen International Center for Mathematics \\
Southern University of Science and Technology \\
1088 Xueyuan Avenue, Shenzhen, Guangdong, China}
\email{yut6@sustech.edu.cn}


\date{7 June, 2024}%{\today}
% \dedicatory{}

\begin{abstract}
We define a map from the skein module of a cusped hyperbolic 3-manifold
to the ring of Laurent series in one variable with integer coefficients that satisfies
two properties:
  its evaluation at peripheral curves coincides with the
  Dimofte--Gaiotto--Gukov 3d-index, and
  it factors through the 3d-quantum trace map associated to a suitable ideal
  triangulation of the manifold.
The map fulfills a supersymmetry prediction of mathematical physics and is part
of a conjectural 3+1 dimensional topological quantum field theory.
\end{abstract}

\maketitle

{\footnotesize
\tableofcontents
}

%%%%%%%%%%%%%%%%%%%%%%%%%%%%%%%%%%%%%%%%%%%%%%%%%%%%%%%%%%%%%%%%%%%%%%%%%%%%
%%%%%%%%%%%%%%%%%%%%%%%%%%%%%%%%%%%%%%%%%%%%%%%%%%%%%%%%%%%%%%%%%%%%%%%%%%%%

\section{Introduction}
\label{sec.intro}

\subsection{Skein modules, quantum trace map and the 3d-index}
\label{sub.3things}

The paper concerns the extension of an interesting power series invariant of knots,
the 3d-index of Dimofte--Gaiotto--Gukov~\cite{DGG1} to the skein module of the knot
complement, using crucially a recently developed 3d-quantum trace map~\cite{GY,PP}. 

Recall that
the skein module $\skein(M)$ of an oriented, connected 3-manifold $M$, introduced
in the early days of quantum topology
by Przytycki~\cite{Przytycki} and Turaev~\cite{Tu:conway}, is a geometrically defined
version of the $\PSL_2(\BC)$-character variety of $\pi_1(M)$. Its spanning set is
the set of unoriented links in $M$ and the local relations are on the one hand
motivated by the recursion relations for the Jones polynomial of a knot in $S^3$,
and on the other hand are deformations of the trace identity for $2 \times 2$ matrices.

The skein module of a 3-manifold plays an important role in topological quantum
field theories (in short, TQFTs) in both $2+1$ and $3+1$ dimensions. Indeed, it is
well-known that the Reshetikhin--Turaev invariant of links in an integer homology
3-sphere $M$ is part of a 2+1 dimensional TQFT and defines a map
\begin{equation}
\label{WRT}
Z_{M,\omega} : \skein(M) \to \BZ[\omega]\end{equation}
for a root of unity $\omega$~\cite{RT,Turaev}. A repackaging of the collection of 
the above maps indexed by complex roots of unity to a single element of the Habiro
ring is possible using Habiro's work~\cite{Habiro:WRT}; see also~\cite{GL:skein}.

The next concept is the 3d-index of Dimofte--Gaiotto--Gukov which roughly speaking
sign-counts states in a supersymmetric field theory using ideas from a 3d-3d
correspondence of mathematical physics, and organizes the count in a Laurent series
in $\indexR$ with integer coefficients~\cite{DGG1,DGG2}. This series
is presented as a lattice sum of products of building blocks (the tetrahedron index),
and depends on a suitable ideal triangulation of a 3-manifold $M$ with torus boundary,
as well as on a choice of an integral homology basis of boundary curves. As was
predicted by physics, this DGG 3d-index was shown to be a topological invariant of
cusped hyperbolic 3-manifolds~\cite{GHRS}. 

The 3d-3d correspondence and its $\calN=2$-supersymmetry algebra predicts not only
the existence of the 3d-index associated to an ideal triangulation $\triang$, but also
to a module $\qglue(\triang)$ of operators which are roughly speaking, a left/right
quotient of a ring of polynomials in $q$-commuting variables (i.e., of a quantum
torus). This was
discussed in several contexts in the physics literature; see for
example~\cite{Dimofte:3dsuper} and~\cite{AGLR}. A mathematical definition of this
module that and its corresponding 3d quantum-trace map was given independently by
Panich--Park and the authors~\cite{GY,PP}, both definitions using as input an ideal
triangulation of the ambient 3-manifold.

But the $\calN=2$-supersymmetry algebra predicts more, namely an extension of the
3d-index to the skein module with two properties:
\begin{itemize}
\item[(a)]
  when evaluated on the peripheral part of the skein module
  coincides with the DGG 3d-index.
\item[(b)]
  it factors through the quantum trace map.
\end{itemize}
This prediction is discussed in detail by Agarwal--Gang--Lee--Romo~\cite{AGLR},
based on earlier work~\cite[Sec.5]{Gang:aspects}.

The goal of the paper is to mathematically fulfill this prediction.
In fact, the map~\eqref{Imap} defined below concerns a 3+1 dimensional aspect of
the skein module $\skein(M)$. Indeed, Witten conjectured the existence of a 3+1
dimensional TQFT whose vector space associated to a 3-manifold $M$ is the skein
module $\skein(M)$. The finiteness of the rank of this module, conjectured by Witten,
was proven by Gunningham--Jordan--Safronov~\cite{Gunningham-Jordan} for closed
3-manifolds. Arithmetic aspects of this TQFT were phrased in terms of holomorphic
quantum modular forms in~\cite{GZ:kashaev}, and in terms of a map from the skein
module of an integer homology 3-sphere to the Habiro ring~\cite[Eqn(3)]{GL:skein}
mentioned above. A full definition of this 3+1 dimensional TQFT is an interesting and
challenging question which, without doubt, lead to a better understanding of 4
dimensional smooth manifolds. 

\subsection{Our results}
\label{sub.results}

We are now ready to formulate our results.

\begin{theorem}
\label{thm.1}
For a cusped hyperbolic 3-manifold $M$ with no non-peripheral $\ints/2$-homology,
there exist a map
\begin{equation}
\label{Imap}    
I_M : \skeinev(M)\to\indexR
\end{equation}
with the following properties:
\begin{itemize}
\item[(a)]
  its restriction at peripheral elements agrees with the 3d-index of
  Dimofte--Gaiotto--Gukov:
\begin{equation}
\label{I2DGG}  
I_M(\qm_\lambda^{-2m}\qm_\mu^{2e}) = I^\DGG_M(m,e) \,.
\end{equation}
\item[(b)]
  it factors through the 3d-quantum trace map for suitable triangulations.
\end{itemize}
\end{theorem}
Here, $\skeinev(M)$ is the even part of the skein module (see
Section~\ref{sec-even-skein}) and the condition that $M$ has no non-peripheral
$\ints/2$-homology (i.e., that $H_1(\partial M;\ints/2)\to H_1(M;\ints/2)$ is
surjective, which for instance is satisfied for complements of links in the 3-sphere)
is a technical assumption that can be relaxed; see Remark~\ref{rem-extra-sum}.

A key ingredient of the above map is the 3d quantum trace map. 
The map $I_M$ and the precise meaning of suitable triangulations is made
clear with the next theorem and the discussion following it. 

\begin{theorem}
\label{thm.2}
For every 1-efficient triangulation $\triang$ of a cusped-hyperbolic manifold with no
non-peripheral $\ints/2$-homology, there exists a map
\begin{equation}
\label{IT}
I_\triang : \qgluev(\triang) \to \indexR
\end{equation}
with the following properties: if $\triang_3 \to \triang_2$ (resp.,
$\triang_2 \to \triang_0$) is a 3--2 (resp., 2--0) Pachner move of 1-efficient
triangulations, there exist maps $\phi_{3,2}$ and $\phi_{2,0}$ that fit in
the commutative diagrams 
\begin{equation}
\label{movediag}
\begin{tikzcd}[row sep=tiny]
& \qgluev(\triang_2) \arrow[rd,"I_{\triang_2}"] \arrow[dd,"\phi_{3,2}"] \\
\skeinev(M) \arrow[ru,"\qtr_{\triang_2}"] \arrow[rd,"\qtr_{\triang_3}"'] &&
\indexR \\
&\qgluev(\triang_3) \arrow[ru,"I_{\triang_3}"']
\end{tikzcd}
\qquad
\begin{tikzcd}[row sep=tiny]
& \qgluev(\triang_0) \arrow[rd,"I_{\triang_0}"] \arrow[dd,"\phi_{2,0}"] \\
\skeinev(M) \arrow[ru,"\qtr_{\triang_0}"] \arrow[rd,"\qtr_{\triang_2}"'] &&
\indexR \\
&\qgluev(\triang_2) \arrow[ru,"I_{\triang_2}"']
\end{tikzcd}
\end{equation}
\end{theorem}

To apply the above theorem and define that map $I_M$, we need a canonical island
of 1-efficient triangulations that are connected by 2--0 and 3--2 Pachner moves.
Such an island exists for cusped hyperbolic 3-manifolds $M$, and the corresponding
triangulations are obtained from regular subdivisions of the canonical ideal cell
decomposition of $M$, after possibly adding tetrahedra to triangulate the faces
of the ideal cells. This is discussed in detail in~\cite[Sec.6]{GHRS}. This and
Theorem~\ref{thm.2} implies that for any 1-efficient triangulation $\triang$ in
that island, the composition 
\begin{equation}
\label{Icomp}
\begin{tikzcd}
  \skeinev(M) \arrow[r,"\qtr_\triang"] &
  \qgluev(\triang) \arrow[r,"I_{\triang}"] & \indexR
\end{tikzcd}
\end{equation}
is independent of $\triang$ and defines the map $I_M$ of Theorem~\ref{thm.1}.

We end this introduction with some remarks on the map~\eqref{Imap}.

\begin{remark}
The image $V^{3d}_M$ of the map~\eqref{Imap} (and consequently, of the map~\eqref{IT})
after tensoring with $\BQ$ is a $\BQ[q^{\pm 1/2}]$-module of finite rank. This follows
from the fact that the
defining formula for the 3d-index is a proper $q$-hypergeometric sum, and hence
$q$-holonomic; see~\cite{WZ} and also~\cite{GL:survey}. On the other hand, it is not
known that the quantum gluing module $\qgluev(\triang)$ has finite rank over
$\BQ(q^{1/2})$. The rank of $V^{3d}_M$ should be related to the rank of the
$\SL_2(\BC)$-Floer Homology defined by C\^{o}t\'{e}--Manolescu~\cite{CM:SL2} as well as
on the size $r$ of the matrices that appear in the holomorphic quantum modular forms
of knot complements~\cite{GZ:kashaev,GZ:qseries}. 
\end{remark}

\begin{remark}
The power series in $q^{1/2}$ that appear in the image of the map~\eqref{Imap}
are not only convergent for $|q|<1$, but also (as is lesser known) for $|q|>1$;
see Remark~\ref{rem.qnot1}. Their asymptotic expansions as $q$ approaches a root
of unity is related to the asymptotic expansion of perturbative complex Chern--Simons
theory and was studied in detail in~\cite{GZ:qseries}.
%Coupled with the supersymmetric calculations of the physicists, in a sense
%the map~\eqref{Imap} implies that the 3d-skein module detects supersymmetry. 
\end{remark}

\begin{remark}
Much like the DGG 3d-index, the map $I_\triang$ is effectively computable, and in
fact it has been computed before it was defined for the complement of the three
simplest hyperbolic knots, namely the $4_1$, $5_2$ and $(-2,3,7)$-pretzel knot
in~\cite{DGG:3knots}. Examples of computation are included in
Section~\ref{sec.compute}. 
\end{remark}

%%%%%%%%%%%%%%%%%%%%%%%%%%%%%%%%%%%%%%%%%%%%%%%%%%%%%%%%%%%%%%%%%%%%%%%%%%%%
%%%%%%%%%%%%%%%%%%%%%%%%%%%%%%%%%%%%%%%%%%%%%%%%%%%%%%%%%%%%%%%%%%%%%%%%%%%%

\section{Skein modules}
\label{sec.skein}

In this section we review the basic properties of the skein module and its
local version (i.e., the stated skein module) and its variants. Note $q^{1/2}$
in~\cite{GY} is $q^{1/8}$ here.

\subsection{Kauffman bracket skein modules}

Let $M$ be an oriented 3-manifold, possibly with boundary. The \term{Kauffman
bracket skein module}, denoted by $\skein(M)$, is the $\ground$-module spanned
by unoriented framed links in $M$ modulo the following skein relations.
\begin{equation}
\label{eq-skein}
\begin{linkdiag}
\fill[gray!20] (-0.1,0)rectangle(1.1,1);
\begin{knot}
\strand[thick] (1,1)--(0,0);
\strand[thick] (0,1)--(1,0);
\end{knot}
\end{linkdiag}
=q^{1/4}
\begin{linkdiag}
\fill[gray!20] (-0.1,0)rectangle(1.1,1);
\draw[thick] (0,0)..controls (0.5,0.5)..(0,1);
\draw[thick] (1,0)..controls (0.5,0.5)..(1,1);
\end{linkdiag}
+q^{-1/4}
\begin{linkdiag}
\fill[gray!20] (-0.1,0)rectangle(1.1,1);
\draw[thick] (0,0)..controls (0.5,0.5)..(1,0);
\draw[thick] (0,1)..controls (0.5,0.5)..(1,1);
\end{linkdiag},\qquad
\begin{linkdiag}
\fill[gray!20] (0,0)rectangle(1,1);
\draw[thick] (0.5,0.5)circle(0.3);
\end{linkdiag}
=(-q^{1/2}-q^{-1/2})
\begin{linkdiag}
\fill[gray!20] (0,0)rectangle(1,1);
\end{linkdiag}.
\end{equation}
Here, each diagram is a portion of the link inside a 3-ball with vertical framing,
and the parts of the links outside the diagrams are the same in each equation.

By inspection, the homology class of the links in $H=H_1(M;\ints/2)$ is unchanged by
the defining relations. Thus, $\skein(M)$ is $H$-graded.

Given a surface, the skein algebra $\skein(\surface)$ is defined as
$\skein(\surface\times[-1,1])$ with the product given by stacking. Our convention is
$\alpha\cup\beta$ stacks $\alpha$ above $\beta$.

We can use a surface to describe a 3-manifold by choosing a Heegaard surface
$i:\surface\embed M$. In other words, $M$ can be recovered by attaching 2- and
3-handles to $\surface\times[-1,1]$. By \cite[Proposition~2.2]{Przytycki:fund}, the
natural map $i_\ast:\skein(\surface)\to\skein(M)$ is surjective, and the kernel is
generated by handle slides. In addition, the map $i_\ast$ behaves as expected on the
homology gradings.

\subsection{Stated skein algebras}

We recall the stated skein algebras defined by L\^e \cite{Le:decomp}. A
\term{punctured bordered surface} is a surface of the form
$\surface=\surclose\setminus\marked$, where $\surclose$ is a compact oriented
surface, possibly with boundary, and $\marked$ is a finite set that intersects every
component of $\partial\surclose$. A component of $\partial\surface$ is called a
\term{boundary edge}, which is homeomorphic to an open interval.

A \term{tangle} over $\surface$ is an unoriented embedded 1-dimensional submanifold
$\alpha\subset\surface\times[-1,1]$ such that $\partial\alpha$ is in
$\partial\surface\times[-1,1]$, and that the heights (i.e. the second coordinates) of
$\partial\alpha$ are distinct over each boundary edge.

A \term{framing} of a tangle $\alpha$ over $\surface$ is a transverse vector field
along $\alpha$ which is vertical at $\partial\alpha$. Here, \term{vertical} means
tangent to $\{\ast\}\times[-1,1]$ in the positive direction. A \term{state} of a
tangle $\alpha$ is a map $\partial\alpha\to\{+,-\}$. By convention, $\pm$ are
identified with $\pm1$ when necessary. Throughout the paper, all tangles will be
framed and stated. An \term{isotopy} of tangles over $\surface$ is homotopy within
the class of tangles over $\surface$. In particular, the height order of
$\partial\alpha$ over each boundary edge of $\surface$ is preserved by isotopy. 

The \term{stated skein module} $\skein(\surface)$ is the $\ground$-module spanned
by stated and framed tangles satisfying the skein relations \eqref{eq-skein} and
following boundary relations.
\begin{align}
\label{eq-arcs}
\relarc{+}{-}&=q^{-1/8}\relempty,\qquad
\relarc{-}{+}=-q^{-5/8}\relempty,\qquad
\relarc{+}{+}=\relarc{-}{-}=0,\\
\label{eq-stex}
\begin{linkdiag}
\fill[gray!20] (0,0)rectangle(1,1);\draw[->] (1,0)--(1,1);
\draw [thick] (0,0.7)..controls(0.8,0.7) and (0.8,0.3)..(0,0.3);
\end{linkdiag}
&=q^{1/8}\relacross{-}{+}-q^{5/8}\relacross{+}{-}.
\end{align}
The same diagram conventions as before apply here. In addition, an arrow on the
boundary of the surface indicates that as one follows the direction of the arrow, the
heights of the endpoints are consecutive and increasing. Although the arrows here
follow the induced orientation on the boundary, an isotopy of the tangle can reverse
it. For example,
\begin{equation}\label{eq-marking-flip}
\relcross{\mu}{\nu}=\relacross[<-]{\nu}{\mu}.
\end{equation}

The homology grading from the last section generalizes to the stated skein algebra
using the relative homology $H_1(\surface,\partial\surface;\ints/2)$. The stated
skein algebra has an additional $\ints$-grading $d_e$ on each bound edge $e$ by the
sum of states.

Although $\skein(\surface)$ has a stacking product $\cup$ as before, for the special
class of surfaces we consider, there is a necessary modification.

A \term{boundary triangle} of a punctured bordered surface $\surface$ is a boundary
component of $\surclose$ containing 3 boundary edges, and a \term{surface with
triangular boundary} is a punctured bordered surface $\surface$ with a distinguished
set of boundary triangles.

Suppose $\surface$ with triangular boundary, we define a new product structure on
$\skein(\surface)$, which we denote by $\cdot$. It is modified from $\cup$ by a power
of $q$. The reason for this modification is to obtain the correct quantization in
Section~\ref{sec-qglue}. Without stating otherwise, we use $\cdot$ as the product for
the rest of the paper.

Let $\omega:\ints^3\otimes\ints^3\to\ints$ be the skew-symmetric bilinear form with
the matrix
\begin{equation}\label{eq-block}
\begin{pmatrix}
0&1&-1\\-1&0&1\\1&-1&0
\end{pmatrix}
\end{equation}
For each boundary triangle $E=\{e_1,e_2,e_3\}$, label the edges cyclically as
Figure~\ref{fig-b-label}. Then the gradings of a tangle $\alpha$ form a vector
$d_E(\alpha)=(d_{e_1}(\a),d_{e_2}(\a),d_{e_3}(\a)) \in\ints^3$. For tangles
$\alpha,\beta$, define a new product by
\begin{equation}
\label{eq-new-prod}
\alpha \cdot \beta=q^{-\frac{1}{8}\sum_{E}\omega(d_E(\alpha),d_E(\beta))}\alpha\cup\beta
\end{equation}
where the sum is over all boundary triangles $E$. It is easy to see that this extends
linearly to an associative product with the empty tangle as unit.

\begin{figure}[htpb!]
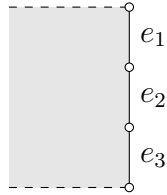

\centering
\begin{linkdiag}[3/2]
\fill[gray!20] (0,0) rectangle (2,3);
\draw[dashed] (0,0) -- (2,0) (0,3) -- (2,3);
\draw (2,0) -- (2,3);
\foreach \y in {0,1,2,3}
\draw[fill=white] (2,\y)circle(0.07);
\foreach \i in {1,2,3}
\path (2,{3.5-\i}) node[right]{$e_\i$};
\end{linkdiag}
\caption{Cyclic labeling of the boundary edges in a boundary triangle.}
\label{fig-b-label}
\end{figure}

\subsection{Corner reduction and splitting homomorphism}

We want to split surfaces into elementary pieces. In the case of punctured bordered
surface, the relation \eqref{eq-stex} introduced by L\^e in \cite{Le:decomp} is key
to make splitting along ideal arcs work. Similarly, we introduced corner reduction in
\cite{GY} to make splitting along closed curves work.

Suppose $\surface$ is a surface with triangular boundary. The \term{corner-reduced
skein module} $\skeincr(\surface)$ is the quotient of $\skein(\surface)$ by the left
ideal (using $\cdot$ product) generated by the following relations near the boundary
triangles.
\begin{equation}\label{eq-crdef}
\relcorner{-}{+}=0,\qquad
\relcorner{-}{-}=-\iunit q^{-1/4},\qquad
\relcorner{+}{+}=\iunit q^{1/4},\qquad
\relcorner{+}{-}=1.
\end{equation}
Note these relations do not preserve the homology or the state sum gradings.

We say a local relation holds at the bottom if it remains true after adding strands
above the one shown. This is slightly stronger than a left ideal since the additional
strands are allowed connect to the existing ones outside the local picture.

\begin{lemma}[{\cite[Corollary~4.13, Lemma~4.14]{GY}}]
\label{lemma-tw}
In $\skeincr(\surface)$, the following relations hold at the bottom near boundary
triangles.
\begin{align}
\label{eq-crhandle}
&&\relaround{1}&=\relaround{0}.\\
\label{eq-twup}
\relbottom{+}&=\iunit q^{\frac{1}{8}(d-d'+3)}\reltwup{-},&
\relbottom{-}&=\iunit q^{\frac{1}{8}(d'-d+3)}\reltwup{+}+q^{\frac{1}{8}(2d''-d-d'+1)}
\reltwup{-}.\\
\label{eq-twdown}
\reltwup{-}&=-\iunit q^{\frac{1}{8}(d'-d-3)}\relbottom{+},&
\reltwup{+}&=q^{\frac{1}{8}(2d''-d-d'-1)}\relbottom{+}-\iunit q^{\frac{1}{8}(d-d'-3)}
\relbottom{-}.
\end{align}
Here, $d$ and $d'$ are the gradings of the left-hand sides on the top and bottom
edges, respectively, and $d''$ is the grading on the third edge in the same boundary
triangle (not shown in the diagrams).
\end{lemma}

Let $\surface$ be a surface with triangular boundary. Given a closed curve $c$ in the
interior of the surface and 3 distinguished points $p_1,p_2,p_3$ on $c$, we can
remove the 3 points from $\surface$ and split $\surface$ along $c$ to obtain a new
surface $\surface_c$ with triangular boundary. Let $p:\surface_c\to\surface$ be the
gluing map.

Let $\alpha$ be a tangle diagram on $\surface$. We can isotope $\alpha$ to be
disjoint from the points $p_i$ and transverse to the curve $c$. If we simply split
the diagram, then the new boundary edges are missing height orders and states. To
resolve the height, we choose an orientation for each arc between the points $p_i$
on $c$, and use the lift of the orientation as the height orders for $\surface_c$.
The states are simply summed.
\begin{equation}
\label{eq-split}
\cut_c(\alpha)=\sum_{s:\alpha\cap c\to\{\pm\}}(\alpha,s).
\end{equation}
Here, $(\alpha,s)$ means for each $x\in\alpha\cap c$, we assign $s(x)$ to the two
endpoints of the split diagram in $p^{-1}(x)$. As an example,
\begin{equation}\label{eq-split-example}
\cut_c\Bigg(
\begin{tikzpicture}[baseline=(ref.base)]
\fill[gray!20] (0,0) rectangle (2,1);
\draw[dashed] (1,0) -- (1,1);
\draw[fill=white] (1,0)circle(0.05) node[below]{\small$p_1$} (1,1)circle(0.05)
node[above]{\small$p_2$};
\draw[thick] (0,0.7) -- (2,0.7) (0,0.3) -- (2,0.3);
\node (ref) at (1,0.5) {\phantom{$-$}};
\end{tikzpicture}
\Bigg)
=\sum_{\mu,\nu=\pm}
\begin{tikzpicture}[baseline=(ref.base)]
\fill[gray!20] (0,0) rectangle (1,1);
\draw (1,0) -- (1,1);
\path[tips,->] (1,0) -- (1,0.95);
\draw[fill=white] (1,0)circle(0.05) (1,1)circle(0.05);
\draw[thick] (0,0.7) -- (1,0.7)\stnode{\mu} (0,0.3) -- (1,0.3)\stnode{\nu};
\node (ref) at (1,0.5) {\phantom{$-$}};
\end{tikzpicture}
\begin{tikzpicture}[baseline=(ref.base)]
\fill[gray!20] (1,0) rectangle (2,1);
\draw (1,0) -- (1,1);
\path[tips,->] (1,0) -- (1,0.95);
\draw[fill=white] (1,0)circle(0.05) (1,1)circle(0.05);
\draw[thick] (1,0.7)\stnodel{\mu} -- (2,0.7) (1,0.3)\stnodel{\nu} -- (2,0.3);
\node (ref) at (1,0.5) {\phantom{$-$}};
\end{tikzpicture} \,.
\end{equation}

\begin{theorem}[{\cite[Theorem~4.16]{GY}}]
$\cut_c$ is a well-defined homomorphism $\skeincr(\surface)\to\skeincr(\surface_c)$.
\end{theorem}

Note when $\surface_c=\surface_1\sqcup\surface_2$ is disconnected,
$\skeincr(\surface_c)$ naturally splits as a tensor product, so the splitting
homomorphism takes the form
\begin{equation}
\skeincr(\surface)\to\skeincr(\surface_1)\otimes\skeincr(\surface_2).
\end{equation}

\subsection{Coalgebra of the annulus}

In addition to being able to split a surface into elementary pieces, another
important use of splitting is to localize complex calculations. Let $\annulus$ denote
the annulus with 3 punctures on both boundary components, so $\annulus$ has
triangular boundary. Choose a simple arc from one side to the other as the standard
arc. When $\annulus$ is split along its core, both components are identified with
$\annulus$ using the standard arc. This splitting homomorphism defines a
comultiplication
\begin{equation}
\Delta:\skeincr(\annulus)\to\skeincr(\annulus)\otimes\skeincr(\annulus).
\end{equation}
Suppose $\surface$ is a surface with triangular boundary. Then similarly, the
splitting homomorphism along a curve parallel to a boundary triangle makes
$\skeincr(\surface)$ a $\skeincr(\annulus)$-comodule.

This is useful because $\skeincr(\annulus)$ also has a counit
\begin{equation}
\epsilon:\skeincr(\annulus)\to\ground
\end{equation}
by \cite[Corollary~4.19]{GY}. This implies that the coaction followed by the counit
is identity. In particular, we can isotope the complicated part of a diagram on
$\surface$ in a neighborhood of a boundary triangle and simplify the diagram using
this strategy. To describe the counit, we use \cite[Theorem~4.18]{GY}, which shows
that there is a spanning set consisting of diagrams that are parallel copies of the
standard arc with some fixed height orders. For the diagram $\alpha$ in
Figure~\ref{fig-std-ann},
\begin{equation}
\epsilon(\alpha)=\begin{cases}
1,&\mu_i=\nu_i\text{ for all $i$},\\0,&\text{otherwise}.
\end{cases}
\end{equation}
Note it is crucial that the arrows are parallel in Figure~\ref{fig-std-ann}.

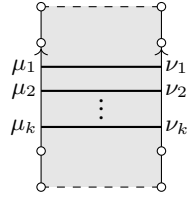
\begin{figure}[htpb!]
\centering
\begin{tikzpicture}[scale=0.8,baseline=(ref.base)]
\fill[gray!20] (0,0) rectangle (2,3);
\draw[dashed] (0,0) -- (2,0) (0,3) -- (2,3);
\draw (0,0) -- (0,3) (2,0) -- (2,3);
\foreach \y in {0,0.6,2.4,3}
\draw[fill=white] (0,\y)circle(0.07) (2,\y)circle(0.07);
\draw[thick] (0,2)\stnodel{\mu_1} -- (2,2)\stnode{\nu_1} (0,1.6)
\stnodel{\mu_2} -- (2,1.6)\stnode{\nu_2} (0,1)\stnodel{\mu_k} -- (2,1)\stnode{\nu_k};
\path (1,1.3)node[rotate=90]{...};
\path[tips,->] (0,0) -- (0,2.3);
\path[tips,->] (2,0) -- (2,2.3);
\end{tikzpicture}
\caption{Diagram $\alpha$ in the annulus $\annulus$.}
\label{fig-std-ann}
\end{figure}

\subsection{Even skein modules}
\label{sec-even-skein}

In this section we recall the properties of the even skein modules. As before,
we fix a compact oriented 3-manifold $M$. The skein module $\skein(M)$ is
graded by the homology $H_1(M;\ints/2)$. The degree 0 part is a version of the even
skein module, but by this definition, the coefficients are still in $\ground$, same
as $\skein(M)$. We can arrange the coefficients to be in $\groundev$ by introducing
more grading restrictions.

As a $\groundev$-module, $\ground$ has a $(\ints/2\times\ints/4)$-grading using
powers of $\iunit$ and $q^{1/8}$. This makes $\groundev$ the degree 0 part of
$\ground$. We identify the counterpart of this grading on framed links.

Note the defining relations \eqref{eq-skein} do not have $\iunit$ or $q^{1/8}$, so we
only need to take care of $q^{1/4}$. This is solved by the framing, since a
Reidemeister I move (adding a full twist in framing) results in a factor of
$-q^{\pm3/4}$. \cite{Bar} describes a $\ints/2$ assignment $\Sigma\mathrm{Spin}$,
abbreviated here by $\sigma$, to framed links depending on a spin structure on $M$.
In particular, Seifert framed knots in a ball have $\sigma=0$. When restricted to
links with even homology, $\sigma$ is independent of the spin structure as well.
Thus, we can define the $(\iunit,q^{1/8})$-degree of a framed link as $(0,2\sigma)$.
Then the even skein module $\skeinev(M)$ is defined as the $\groundev$-submodule
where all gradings are zero, and \eqref{eq-skein} can be restricted to $\skeinev(M)$
by \cite{Bar}.

Now suppose $\surface$ is a punctured bordered surface. The even skein algebra
$\skeinev(\surface)$ similarly requires zero homology class in
$H_1(\surface,\partial\surface;\ints/2)$, but we also need to describe framing and
state requirements for tangles. Let $\alpha$ be a tangle over $\surface$ with even
homology. Then on each boundary edge, the number of endpoints is even. By an isotopy,
we require the diagram near the boundary edge have positive height order, i.e.\ like
the order in the defining relation \eqref{eq-stex}. We can pair up the endpoints and
connect like \eqref{eq-stex} without introducing new crossings. Let $\bar{\alpha}$ be
a link obtained this way. Then we define $\sigma(\alpha)=\sigma(\bar{\alpha})$. Again
using the argument of \cite{Bar}, $\sigma(\alpha)$ is independent of how
$\bar{\alpha}$ is connected. Define two more quantities: $2n(\alpha)$ is the number
of endpoints of $\alpha$, and $2d(\alpha)$ is the sum of all states of $\alpha$. $n$
and $d$ are both integers when $\alpha$ has even homology. Now define the
$(\iunit,q^{1/8})$-degree of $\alpha$ as $(d(\alpha),2\sigma(\alpha)-n(\alpha))$.
Then the even skein algebra $\skeinev(\surface)$ is again defined as the
$\groundev$-submodule where all gradings are zero. More explicitly,
\begin{equation}\label{eq-skeinev-def}
\skeinev(\surface)=\Span_{\groundev}
\{\iunit^{d(\alpha)}q^{n(\alpha)/8+\sigma(\alpha)/4}\alpha\mid
\alpha\text{ has even homology}\}.
\end{equation}
It is easy to check that the defining relations \eqref{eq-skein}--\eqref{eq-stex}
restricts to $\skeinev(M)$.

If $\surface$ is a dual surface of the 3-manifold $M$, then the natural map
$\skein(\surface)\onto\skein(M)$ is still surjective on the even part $\skeinev$.
Although the $\surface$-diagram of an even link in $M$ may not have even homology,
the difference in homology is generated by $A,B$-circles, so handle slides bring the
diagram to even.

Finally, we deal with the corner reduction. Since the gradings above do not easily
descend to $\skeincr(\surface)$, we define the even part $\skeincrev(\surface)$
simply as the image of $\skeinev(\surface)$ in the quotient
$\skein(\surface)\onto\skeincr(\surface)$.

\begin{lemma}\label{lemma-split-ev}
Let $\surface_c$ be $\surface$ split along a closed curve $c$. Then the image of
$\skeincrev(\surface)$ under the splitting
$\skeincr(\surface)\to\skeincr(\surface_c)$ is in the even part
$\skeincrev(\surface_c)$.
\end{lemma}

\begin{proof}
Recall 3 points are removed from $c$ to make $\surface_c$. Let $\surface'$ be
$\surface$ minus these 3 points. Then $\surface_c$ is $\surface'$ split along 3 ideal
arcs. By the same handle slide argument as above, a tangle with even homology can be
isotoped so that it is even on $\surface'$. Then it is easy to see that the split
diagram has even homology on $\surface_c$.

Next we consider the $\iunit$-degree $d$. By definition, the change in $d$ after
splitting comes from the extra states created by the splitting. However, the matching
states of splitting and the even intersection with $c$ shows that the change in $d$
is even.

Finally, we deal with the $q^{1/8}$-degree $2\sigma-n$. When done with diagrams, the
height order on each pair of new boundary edges have opposite orientations. See
\eqref{eq-split-example}. This means an isotopy is required to twist the side with
negative order to positive. After pairing the endpoints and connecting them, the
twist from the negative side becomes a Reidemeister I move for each pair of
endpoints, as shown in the following example.
\begin{equation}
\begin{tikzpicture}[baseline=(ref.base)]
\fill[gray!20] (0,0) rectangle (1,1.5);
\draw (1,0) -- (1,1.5);
\path[tips,->] (1,0) -- (1,1.45);
\draw[fill=white] (1,0)circle(0.05) (1,1.5)circle(0.05);
\draw[thick] foreach \y in {0.3,0.6,0.9,1.2} {(0,\y) -- (1,\y)};
\node (ref) at (1,0.75) {\phantom{$-$}};
\end{tikzpicture}
\begin{tikzpicture}[baseline=(ref.base)]
\fill[gray!20] (1,0) rectangle (2,1.5);
\draw (1,0) -- (1,1.5);
\path[tips,->] (1,0) -- (1,1.45);
\draw[fill=white] (1,0)circle(0.05) (1,1.5)circle(0.05);
\draw[thick] foreach \y in {0.3,0.6,0.9,1.2} {(1,\y) -- (2,\y)};
\node (ref) at (1,0.75) {\phantom{$-$}};
\end{tikzpicture}
\quad=\quad
\begin{tikzpicture}[baseline=(ref.base)]
\fill[gray!20] (0,0) rectangle (1,1.5);
\draw (1,0) -- (1,1.5);
\path[tips,->] (1,0) -- (1,1.45);
\draw[fill=white] (1,0)circle(0.05) (1,1.5)circle(0.05);
\draw[thick] foreach \y in {0.3,0.6,0.9,1.2} {(0,\y) -- (1,\y)};
\node (ref) at (1,0.75) {\phantom{$-$}};
\end{tikzpicture}
\begin{tikzpicture}[baseline=(ref.base)]
\fill[gray!20] (1,0) rectangle (2,1.5);
\draw (1,0) -- (1,1.5);
\path[tips,->] (1,1.5) -- (1,0.05);
\draw[fill=white] (1,0)circle(0.05) (1,1.5)circle(0.05);
\begin{knot}[background color=gray!20,clip width=5,end tolerance=5pt,clip radius=0.1cm]
\strand[thick] (2,1.2)..controls +(-1,0) and +(0.4,0)..(1,0.3);
\strand[thick] (2,0.9)..controls +(-0.6,0) and +(0.4,0)..(1,0.6);
\strand[thick] (2,0.6)..controls +(-0.4,0) and +(0.6,0)..(1,0.9);
\strand[thick] (2,0.3)..controls +(-0.2,0) and +(0.8,0)..(1,1.2);
\end{knot}
\node (ref) at (1,0.75) {\phantom{$-$}};
\end{tikzpicture}
\quad\to\quad
\begin{tikzpicture}[baseline=(ref.base)]
\fill[gray!20] (0,0) rectangle (1,1.5);
\draw (1,0) -- (1,1.5);
\path[tips,->] (1,0) -- (1,1.45);
\draw[fill=white] (1,0)circle(0.05) (1,1.5)circle(0.05);
\draw[thick] foreach \y in {0.3,0.9} {(0,\y)..controls +(1,0) and +(1,0)..+(0,0.3)};
\node (ref) at (1,0.75) {\phantom{$-$}};
\end{tikzpicture}
\begin{tikzpicture}[baseline=(ref.base)]
\fill[gray!20] (0.5,0) rectangle (2,1.5);
\draw (0.5,0) -- (0.5,1.5);
\path[tips,->] (0.5,1.5) -- (0.5,0.05);
\draw[fill=white] (0.5,0)circle(0.05) (0.5,1.5)circle(0.05);
\begin{knot}[background color=gray!20,clip width=5,end tolerance=5pt,clip radius=0.1cm]
  \strand[thick] (2,1.2)..controls +(-1,0)
  and +(0.4,0)..(1,0.3)..controls +(-0.4,0) and +(-0.4,0)..(1,0.6);
\strand[thick] (2,0.9)..controls +(-0.6,0) and +(0.4,0)..(1,0.6);
\strand[thick] (2,0.6)..controls +(-0.4,0)
and +(0.6,0)..(1,0.9)..controls +(-0.4,0) and +(-0.4,0)..(1,1.2);
\strand[thick] (2,0.3)..controls +(-0.2,0) and +(0.8,0)..(1,1.2);
\end{knot}
\node (ref) at (1,0.75) {\phantom{$-$}};
\end{tikzpicture}
\end{equation}
Thus, the change in the framing and the number of endpoints cancel to preserve the
$q^{1/8}$ degree.
\end{proof}

\begin{lemma}
\label{lemma-counit-ev}
$\skeincrev(\annulus)$ is the $\groundev$-span of diagrams of the form
Figure~\ref{fig-std-ann} with even number of components. As a result, the image of
$\skeincrev(\annulus)$ under the counit $\epsilon:\skeincr(\annulus)\to\ground$ is
$\groundev$.
\end{lemma}

\begin{proof}
The second part follows from the first since the counit of Figure~\ref{fig-std-ann}
is $0$ or $1$, so we focus on the first part.

\cite[Appendix~A]{GY} described how to reduce tangle diagrams on $\annulus$. By
applying the defining relations \eqref{eq-skein}--\eqref{eq-stex} and the relations
from Lemma~\ref{lemma-tw} in a specific direction, any diagram can be reduced to the
form in Figure~\ref{fig-std-ann}. As mentioned before, the defining relations
preserve all gradings as mentioned before.

The rest of the relations change the homology, but the evenness of the original
diagram implies that if any of them applies to the diagram, then another relation of
the same type also applies immediately afterwards. Therefore, we can combine these
relations in a way the preserves the homology class. After careful calculations with
the coefficients, we see that the combinations also preserve all gradings. Here we
give one example.
\begin{equation}
\begin{linkdiag}[0.8]
\fill[gray!20] (0,0) rectangle (1,1.6); \draw (1,0) -- (1,1.6);
\draw[fill=white] (1,0)circle(0.07) (1,0.8)circle(0.07);
\draw[thick] (0,1.1) -- +(1,0) \stnode{-} (0,1.4) -- +(1,0) \stnode{-};
\end{linkdiag}
=-\iunit q^{\frac{1}{8}(d'-d-3)}
\begin{linkdiag}[0.8]
\fill[gray!20] (0,0) rectangle (1,1.6); \draw (1,0) -- (1,1.6);
\draw[fill=white] (1,0)circle(0.07) (1,0.8)circle(0.07);
\draw[thick] (0,1.4) -- +(1,0) \stnode{-} (0,1.1)..controls+(0.5,0)
and +(-0.5,0)..(1,0.5)\stnode{+};
\end{linkdiag}
=(-\iunit q^{\frac{1}{8}(d'-d-3)})(-\iunit q^{\frac{1}{8}((d'+1)-(d+1)-3)})
\begin{linkdiag}[0.8]
\fill[gray!20] (0,0) rectangle (1,1.6); \draw (1,0) -- (1,1.6);
\draw[fill=white] (1,0)circle(0.07) (1,0.8)circle(0.07);
\begin{knot}[clip radius=0.2cm]
\strand[thick] (0,1.1)..controls+(0.5,0) and +(-0.5,0)..(1,0.5)\stnode{+};
\strand[thick] (0,1.4)..controls+(0.5,0) and +(-0.5,0)..(1,0.2)\stnode{+};
\end{knot}
\end{linkdiag}
\end{equation}
The height order is positive throughout, and the relations hold at the bottom. The
meaning of $d,d'$ is the same as in Lemma~\ref{lemma-tw}. The coefficient combines
into $-q^{\frac{1}{8}(2d'-2d-6)}$, which compensates the framing change introduced by
the crossing. The $\iunit$-grading also changes by $+2$, which is even.

Therefore, there is a set of reduction rules that preserve all gradings, and any even
diagram is reduced to the form in Figure~\ref{fig-std-ann}. This proves the lemma.
\end{proof}

%%%%%%%%%%%%%%%%%%%%%%%%%%%%%%%%%%%%%%%%%%%%%%%%%%%%%%%%%%%%%%%%%%%%%%%%%%%%
%%%%%%%%%%%%%%%%%%%%%%%%%%%%%%%%%%%%%%%%%%%%%%%%%%%%%%%%%%%%%%%%%%%%%%%%%%%%

\section{Quantum trace map}
\label{sec.qtrace}

As mentioned in the introduction, the 3d-quantum trace map plays a key role in
the proof of Theorems~\ref{thm.1} and~\ref{thm.2}. In this section we review its
definition and its properties following the notation of our previous work~\cite{GY}.

\subsection{Dual surface from triangulation}
\label{sec-dualS}

Let $T$ be an oriented tetrahedron. A labeling of the vertices of $T$ by $0,1,2,3$ is
compatible with the orientation if vertices $1,2,3$ are counterclockwise when viewed
from vertex $0$. The tetrahedron has a dual surface $\lantern$ which is a sphere with
4 boundary components. Borrowing from the theory of mapping class groups of surfaces,
we call it the \term{lantern}. The boundary circles of the lantern are labeled
according to the vertex opposite to it. The lantern also comes with 6 \term{standard
arcs} dual to the edges of $T$. The embedding of the lantern and a top view are given
in the first two parts of Figure~\ref{fig-smoothL}, where the blue arcs are the
standard arcs. The last part is the lantern laid flat in the plane. Note the
reordering of the labels to maintain the orientation.

\begin{figure}[htpb!]
\centering
\begin{tikzpicture}[baseline=0.5cm]
\path (0,2.5)coordinate(A) (-2.25,-0.5)coordinate(B) (-0.25,-1.75)
coordinate(C) (2.25,-0.5)coordinate(D);
\path (A)node[right]{0} (B)node[below]{1} (C)node[anchor=30]{2} (D)node[below]{3};
\draw[dashed] (B) -- (D);
\fill[gray!20] (0.4,1) arc[radius=0.4,start angle=0,end angle=180]
arc[radius=0.6,start angle=0,end angle=-90]
arc[radius=0.4,start angle=90,end angle=-90]
arc[radius=0.6,start angle=90,end angle=0]
arc[x radius=0.4,y radius=0.25,start angle=-180,end angle=0]
arc[radius=0.6,start angle=180,end angle=90]
arc[radius=0.4,start angle=-90,delta angle=-180]
arc[radius=0.6,start angle=-90,end angle=-180];
\begin{scope}[red,thick]
\draw[fill=gray!10] (-1,0)circle(0.4) (1,0)circle(0.4);
\draw (0.4,1)arc[radius=0.4,start angle=0,end angle=180] (0.4,-1)
arc[x radius=0.4,y radius=0.25,start angle=0,end angle=-180];
\draw[dashed] (0.4,1)arc[radius=0.4,start angle=0,end angle=-180] (0.4,-1)
arc[x radius=0.4,y radius=0.25,start angle=0,end angle=180];
\end{scope}
\begin{scope}[blue,thick]
\draw (-1,0.4)arc[radius=0.6,start angle=-90,end angle=0]
(1,0.4)arc[radius=0.6,start angle=-90,end angle=-180]
(-1,-0.4)arc[radius=0.6,start angle=90,end angle=0]
(1,-0.4)arc[radius=0.6,start angle=90,end angle=180];
\path (-1,0)++(-30:0.4)coordinate(E) (1,0)++(-150:0.4)coordinate(F);
\draw (E) to[out=30,in=150] (F);
\draw[dashed] (0,1) ++(-85:0.4) to[out=-95,in=90] (0,-0.75);
\end{scope}
\draw[knot,knot gap=7,background color=white] (A) -- (C);
\draw (A) -- (B) -- (C) -- (D) -- cycle;
\end{tikzpicture}
$\Longrightarrow$\quad
\begin{tikzpicture}[baseline=(ref.base)]
\draw[thick,blue,fill=gray!20] (0,0)circle(1);
\fill[white] (0,0)circle(0.12) foreach \t in {90,-30,-150} {(\t:1)circle(0.12)};
\draw[thick,blue,dashed] foreach \t in {-30,-150,90} {(\t:0.12) -- (\t:1.12)};
\draw[thick,red,radius=0.12] (90:1)circle node[anchor=-90]{\small2}
(-150:1)circle node[anchor=30]{\small3}
(-30:1)circle node[anchor=150,outer sep=2pt]{\small1}
(0,0)circle node[anchor=-60,outer sep=1pt]{\small0};
\draw[knot,knot gap=7,background color=white]
(0,0) -- (-90:1.5)node[anchor=90]{2}
(0,0) -- (30:1.5)node[anchor=-150]{3}
(0,0) -- (150:1.5)node[anchor=-30]{1};
\draw (0,0) circle(1.5) node[anchor=120,outer sep=1pt]{0};
\node (ref) at (0,0){\phantom{$-$}};
\end{tikzpicture}
\quad$\Longrightarrow$\quad
\begin{tikzpicture}[baseline=(ref.base)]
\fill[gray!20] (0,0)circle(1.5);
\draw[thick,blue] (0,0)circle(1);
\draw[thick,blue] foreach \t in {-30,-150,90} {(0,0) -- (\t:1)};
\fill[white] (0,0)circle(0.12) foreach \t in {-30,-150,90} {(\t:1)circle(0.12)};
\draw[thick,red,radius=0.12] (90:1)
circle node[anchor=-90,outer sep=2pt]{1} (-30:1)
circle node[anchor=150]{2} (-150:1)circle node[anchor=30]{3} (0,0)
circle node[anchor=90,outer sep=2pt]{0};
\path[inner sep=2pt] (-90:1)node[below]{$\qz$} (150:1)
node[anchor=-40,inner sep=1pt]{$\qz'$} (30:1)node[anchor=-150]{$\qz''$};
\node (ref) at (0,0){\phantom{$-$}};
\end{tikzpicture}
\caption{Lantern surface in a tetrahedron.}
\label{fig-smoothL}
\end{figure}
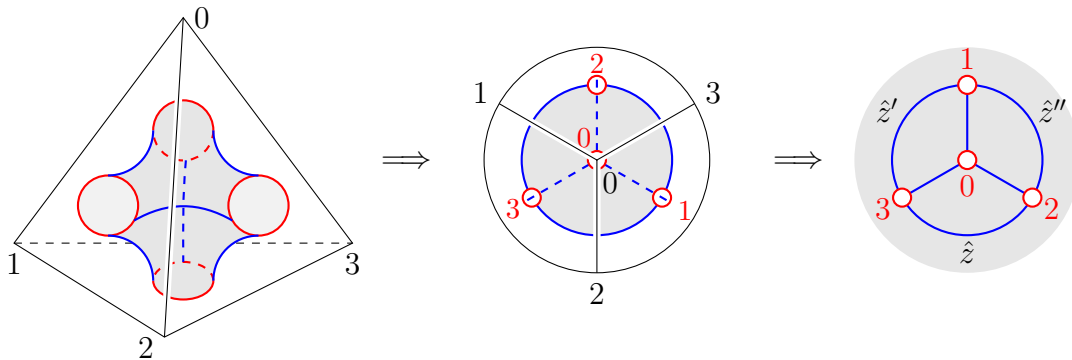

Suppose $M$ is a compact oriented 3-manifold possibly with boundary. Given a
collection of tetrahedra $\triang=\{T_1,\dotsc,T_N\}$ and orientation reversing face
pairing, if the space obtained from gluing the tetrahedra minus the vertices is
homeomorphic to the interior of $M$ preserving the orientation, then $\triang$ is an
oriented triangulation of $M$.

Given an ideal triangulation $\triang$, there is an embedded (Heegaard) surface
$\surface_\triang\subset M$ glued from the lanterns in the tetrahedra. The standard
arcs around the same edge $e$ in the triangulation are connected to form a curve
$B_e$ when the lanterns are glued. On the other hand, each face $f$ of the
triangulation contains a circle $A_f$ which is a boundary circle of the adjacent
lantern. The data $\heeg_\triang:=(\surface_\triang,\{A_f\},\{B_e\})$ is called the
\term{dual surface} of $\triang$. It is equivalent to the triangulation $\triang$
since the intersection pattern of the curves $\{A_f\}$ and $\{B_e\}$ determines the
face pairings of the triangulation. Also note that $M$ can be obtained directly from
$\heeg_\triang$. This can be done by gluing 2-handles to
$\surface_\triang\times[-1,1]$ along $A_f\times\{-1\}$ and $B_e\times\{1\}$ and then
capping off spherical boundary components as necessary. Note the $A$-circles
determine the handlebody inside the dual surface.

\subsection{Quantum tori}

Let us briefly recall this notion of Laurent polynomials in $q$-commuting variables.
A quantum torus on generators $x_1,\dotsc,x_r$ over an algebra $R$ is an associative
$R$-algebra with unit of the form
\begin{equation}
  \qtorus\langle x_1,\dotsc,x_r\rangle:=R\langle x_1^{\pm1},\dotsc,x_r^{\pm1}
  \rangle/\ideal{x_ix_j-q^{D_{ij}}x_jx_i}
\end{equation}
for some skew-symmetric $r\times r$ matrix $D$. It is a free $R$-module with basis
\begin{equation}
\label{eq-Weyl-order}
[\vec{x}^k]=q^{-\frac{1}{2}\sum_{i<j}D_{ij}k_ik_j}x_1^{k_1}\dotsm x_r^{k_r},
\qquad k=(k_1,\dotsc,k_r)\in\ints^r.
\end{equation}
These monomials are called \term{Weyl-ordered}, which is designed to be independent
of the order of the generators. This normalization can be formally understood as the
Baker-Campbell-Hausdorff formula for $e^{\sum_i k_i\ln(x_i)}$. Note that
Weyl-ordering can be defined for all products with $q$-commuting factors.

The product on the quantum torus can be described using the Weyl-ordered monomial
basis as
\begin{equation}
\label{eq-qtorus-prod}
[\vec{x}^k][\vec{x}^l]=q^{\frac{1}{2}\omega(k,l)}[x^{k+l}]
=q^{\omega(k,l)}x^lx^k,
\end{equation}
where $\omega(k,l)=k^tDl$ is the bilinear form associated to $D$.

\subsection{Quantum gluing module}
\label{sec-qglue}

We next recall the quantum gluing module associated to an ideal triangulation.
First consider a single tetrahedron. The quantum gluing module $\qglue(\lantern)$ is
a quotient of the quantum torus over $\ground$
\begin{equation}
\label{eq-qglue-lantern}
  \qglue(\lantern)=\qtorus\langle\qz,\qz',\qz''\rangle/
  \lideal{\qz^{-2}+\qz''^2-1,[\qz\qz'\qz'']-\iunit q^{1/4}},
\end{equation}
where the $q$-commuting relations are given by $\frac{1}{4}\omega$ from
\eqref{eq-block}. Explicitly,
\begin{equation}
\label{eq-zzp}
\qz\qz'=q^{1/4}\qz'\qz,\qquad
\qz'\qz''=q^{1/4}\qz''\qz,\qquad
\qz''\qz=q^{1/4}\qz\qz''.
\end{equation}
Note that $[\qz\qz'\qz'']$ is central, so we can eliminate one of the generators,
typically $\qz'$, from the quantum torus. We will freely do so when convenient.

The variables $\qz,\qz',\qz''$ correspond to the 01, 02, 03 edges, respectively, and
opposite edges are assigned the same generators. It is easy to check that orientation
preserving symmetries of $\qglue(\lantern)$ induce cyclic permutations of
$\qz,\qz',\qz''$ so the definition is independent of the labeling of the tetrahedron.

\begin{remark}
\label{rem-extra-gen}
In \cite{GY}, $\qglue(\lantern)$ has an extra generator $\qy$ satisfying
$\qy^2=\qz^2$. In our main consideration of 3d-index, only squares appear, so the
distinction is irrelevant.
\end{remark}

Let $\triang$ be an oriented triangulation of an oriented 3-manifold $M$ with torus
boundary. Throughout, $N$ is the number of tetrahedra, $r$ is the number of boundary
components, $i$ index edges, $j$ index tetrahedra. In the notations of the ideal, the
range of indices are omitted when clear from context.

The quantum gluing module $\qglue(\triang)$ is defined as
\begin{equation}
\label{eq-qglue-def}
\qglue(\triang):=\rideal{\qe_i+q^{1/2}}\backslash
\Big(\bigotimes_{j=1}^N\qtorus\langle\qz_j,\qz''_j\rangle\Big)
\quotbyL{\qz_j^{-2}+\qz''^2_j-1},
%\rideal{\qe_i+q^{1/2}}\backslash\Big(\bigotimes_{j=1}^N\qglue(\lantern_j)\Big)
\end{equation}
where $\qe_i$ is the Weyl-ordered product of $\qz^\square_j$ around the $i$-th edge.
The monomials $\qe_i$ commute with each other by \cite[Theorem~4.1]{Neu}. Moreover,
$N-r$ of them, say $\qe_1,\dotsc,\qe_{N-r}$, are independent, and the rest can be
written as monomials in $\qe_1,\dotsc,\qe_{N-r}$. The argument of
\cite[Remark~4.6]{GHRS} shows that the independent $N-r$ edges also generate the same
right ideal.

\subsection{Quantum trace map}

The quantum trace map
\be
\label{qtr-def}
\qtr_\triang:\skein(M)\to\qglue(\triang)
\ee
is defined by the following diagram:
\begin{equation}
\label{eq-qtr-def}
\begin{tikzcd}
\skein(\surface_\triang) \arrow[d,two heads] \arrow[r,"\cut_A"] &
\displaystyle\bigotimes_{j=1}^N\skeincr(\lantern_j) \arrow[r,two heads] &
\displaystyle\bigotimes_{j=1}^N\qglue(\lantern_j) \arrow[d,two heads] \\
\skein(M) \arrow[rr,dashed,"\qtr_\triang"]
&& \qglue(\triang)
\end{tikzcd} \,.
\end{equation}
Here $\cut_A$ is the splitting map along all $A$-circles, and
$\skeincr(\lantern)\onto\qglue(\lantern)$ sends a standard arc with $-$ states at
both endpoints to the corresponding generator $\qz^\square$. More generally, the
standard arcs with states $\mu,\nu$ is sent to $\delta_{\mu\nu}(\qz^\square)^{-\mu}$.

\begin{remark}
Suppose $M$ has a boundary component with Euler characteristic $\chi\ne0$, then the
quantum gluing module $\qglue(\triang)$ has quantum inconsistency. There is a product
of $\qe_i$ which is 1 in the quantum torus but $q^{\chi/2}$ according to the
quotient. Although quantum trace map is still defined in this case, most
applications, such as 3d-index considered in this paper, require consistency.
Therefore, we only consider torus boundary in this paper.
\end{remark}

\subsection{Even gluing modules}

We now introduce an even version of the quantum gluing module. 
The even gluing module $\qgluev(\triang)$ is defined as the $\groundev$-span of
squares of Weyl-ordered monomials. We adopt the convention that an upper case letter is
the square of the lower case letter, e.g.\ $\qZ=\qz^2$. Note the coefficients in the
following relations are in $\groundev$.
\begin{equation}
\qZ''\qZ=q^{1/2}[\qZ\qZ'']=q\qZ\qZ'',\qquad
[\qZ\qZ'\qZ'']=-q^{1/2}.
\end{equation}

Let $\qtorusev$ denote a quantum torus over $\groundev$. Then $\qgluev(\triang)$ is a
quotient of $\bigotimes_j\qtorusev\langle\qZ_j,\qZ''_j\rangle$. Our goal is to find a
presentation based on this quotient.

Recall $\qe_i$ commute with each other, and they generate a lattice $\Lambda$ of rank
$N-r$ (with multiplication as the group operation). $\Lambda$ can be embedded in
$\ints^{2N}$ by taking exponents after eliminating $\qz'_j$ using the vertex
equation. This is the same as taking the rows of the Neumann--Zagier matrix.

\begin{lemma}
\label{lemma-basis-change}
Suppose $M$ has no non-peripheral $\ints/2$-homology. Then there exists a generating
set of monomials
$x_1,\dotsc,x_{2N}\in\bigotimes_{j=1}^N\qtorus\langle\qz_j,\qz''_j\rangle$ such that
$x_1,\dotsc,x_{N-r}$ is an independent subset of $\qe_1,\dotsc,\qe_N$.
\end{lemma}

\begin{proof}
$\ints^{2N}$ has a symplectic form that corresponds to the $q$-commuting relations of
$\qz_j,\qz''_j$. By \cite[Theorem~4.2]{Neu}, $\Lambda^\perp=\Lambda$ when $M$ has no
non-peripheral $\ints/2$-homology. Here, the complement is defined using the
symplectic form. This shows that $\Lambda$ is a direct summand of $\ints^{2N}$. Thus,
we can make a change of basis so that $x_1,\dotsc,x_{N-r}$ is a basis of
$\bar{\Lambda}$, and $x_{N-r+1},\dotsc,x_{2N}$ is the rest of the basis of
$\ints^{2N}$. The lemma is just a translation of this basis.
\end{proof}

\begin{proposition}
\label{prop-qgluev}
If $M$ has no non-peripheral $\ints/2$-homology, then the even part
has the presentation
\begin{equation}
\label{eq-qgluev}
\qgluev(\triang)=\rideal{\qE_i-q}\backslash\Big(
\bigotimes_{j=1}^N\qtorusev\langle\qZ_j,\qZ''_j\rangle\Big)
\quotbyL{\qZ_j^{-1}+\qZ_j''-1}.
\end{equation}
\end{proposition}

If necessary, we can restore $\qZ'_j$ in the presentation and add the (central)
relation $[\qZ_j\qZ'_j\qZ''_j]+q^{1/2}$.

\begin{proof}
For convenience, write $\qtorusev_\triang$ for the tensor quantum torus. We need to
show
\begin{equation}
\label{eq-qgluev-cap}
(\rideal{\qe_i+q^{1/2}}+\lideal{\qz_j^{-2}+\qz_j''^2-1})\cap\qtorusev_\triang
=\rideal{\qE_i-q}+\lideal{\qZ_j^{-1}+\qZ_j''-1}.
\end{equation}
Using the generators from Lemma~\ref{lemma-basis-change} and the monomial basis of
the quantum torus, we can write an element in the left-hand side of
\eqref{eq-qgluev-cap} as

\begin{equation}
\label{eq-qgluev-ielem}
\sum_{i=1}^{N-r}(x_i+q^{1/2})\sum_{k\in\ints^{2N-r}}c_{ik}\tilde{x}^k+
\sum_{j=1}^N\sum_{k\in\ints^{2N-r}}d_{jk}\tilde{x}^k
(\qz_j^{-2}+\qz_j''^2-1)\in\qtorusev_\triang,
\end{equation}
where $c_{ik},d_{jk}\in\ground[x_1,\dotsc,x_{N-r}]$, and
$\tilde{x}^k=[x_{N-r+1}^{k_1}\dotsm x_{2N}^{k_{2N-r}}]$ is the Weyl-ordered monomial.
Since the variables $x_{N-r+1},\dotsc,x_{2N}$ only appear in $\tilde{x}^k$ and
$\qz_j,\qz''_j$, terms where $k\notin(2\ints)^{2N-r}$ must cancel for this element to
be in $\qtorusev_\triang$, so we may restrict the sum to $k\in(2\ints)^{2N-r}$. Using
long division, we can extract all $x_i+q^{1/2}$ from $d_{jk}$, so we can assume
$d_{jk}$ are constants.

Let $\qtorus^2_\triang$ be the $\ground$-span of squares, or
$\qtorus^2_\triang=\bigoplus_{s=0}^1\bigoplus_{t=0}^3\iunit^sq^{t/8}\qtorusev_\triang$.
In \eqref{eq-qgluev-ielem}, the second term is in $\qtorus^2_\triang$ by assumption,
so the first term is in $\qtorus^2_\triang$ as well. This means
\begin{equation}
\sum_{i=1}^{N-r}(x_i+q^{1/2})c_{ik}
\in\qtorus^2_\triang\cap\ground[x_1,\dotsc,x_{N-r}]
=\ground[x_1^2,\dotsc,x_{N-r}^2].
\end{equation}
This shows we can rearrange the sum to be $\sum_{i=1}^{N-r}(x_i^2-q)c'_{ik}$ for
$c'_{ik}\in\ground[x_1^2,\dotsc,x_{N-r}^2]$. Putting this into \eqref{eq-qgluev-ielem},
we get
\begin{equation}
\label{eq-qgluev-ielem2}
\sum_{k\in(2\ints)^{2N-r}}\Big(\sum_{i=1}^{N-r}(x_i^2-q)c'_{ik}\tilde{x}^k
+\sum_{j=1}^Nd_{jk}\tilde{x}^k(\qz_j^{-2}+\qz_j''^2-1)\Big)
\in\qtorusev_\triang\subset\qtorus^2_\triang.
\end{equation}
Finally, the $(\ints/2\times\ints/4)$-grading over $\groundev$ by powers of $\iunit$
and $q^{1/8}$ can be used to rearrange the coefficients in \eqref{eq-qgluev-ielem2}
into $\groundev$. This proves~\eqref{eq-qgluev-cap}.
\end{proof}

\begin{remark}
\label{rmk-qgluev}
The same condition on homology appears in \cite{HRS}. When the condition is not
satisfied, it is easy to find extra relations in the even part. We use the example of
the hyperbolic census manifold \texttt{m136} considered in \cite{HRS}. It has 1 torus
boundary and homology $\ints\oplus(\ints/2)^2$, so $H_1(\partial M;\ints/2)\to
H_1(M;\ints/2)$ cannot be surjective. Using the default triangulation given by
\texttt{SnapPy} with isomorphism signature \texttt{eLMkbcdddhhqqa}, we see there is
an edge $\qe_1=\qZ_1\qZ_2\qZ_3\qZ_4$ already in the even part. This means the even
part contains the relation $\qe_1=-q^{1/2}$, while the naive presentation
\eqref{eq-qgluev} only contains the weaker condition $\qE_1=q$. There are refinements
of the quantum gluing modules that reduce unexpected evenness but not completely.

This is also reflected in the even skein modules. Suppose $\surface$ is a dual
surface of $M$ and consider two even $\surface$-diagrams for an even link in $M$.
Although they are related by handle slides, the intermediate stages may not be even
on $\surface$. The naive solution is to only slide along curves twice in a row, but
when there is non-peripheral $\ints/2$-homology, these moves do not relate all
even diagrams. The dual surface to the triangulation above easily produces such an
example.
\end{remark}

\subsection{The even part of quantum trace}

Having introduced the even skein module and the even gluing module, we now
come to the even quantum trace map. 

\begin{theorem}
The quantum trace restricts to a map $\skeinev(M)\to\qgluev(\triang)$.
\end{theorem}

\begin{proof}
Since the quantum trace map is induced by $\skein(\surface)\to\qglue(\triang)$, we
only need to show that $\skeinev(\surface)$ maps to $\qgluev(\triang)$. By
Lemma~\ref{lemma-split-ev}, this is further reduced to the statement that
$\skeincr(\lantern)\onto\qglue(\lantern)$ maps $\skeincrev(\lantern)$ to
$\qgluev(\lantern)$.

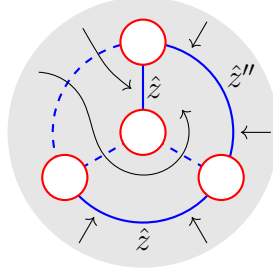
\begin{figure}[htpb!]
\centering
\begin{tikzpicture}
\tikzmath{\r=1.2;\s=0.3;\d=0.45;\e=\r*0.6;}
\fill[gray!20] (0,0)circle(\r+0.6);
\begin{scope}[thick,blue,radius=\r]
\draw[dashed] (-150:\r) arc[start angle=-150,delta angle=-120];
\draw (90:\r) arc[start angle=90,delta angle=-240];
\draw (0,0) -- (90:\r);
\draw[dashed] (0,0) -- (-30:\r) (0,0) -- (-150:\r);
\end{scope}
\draw[thick,red,fill=white] (90:\r)circle(\s) (-30:\r)
circle(\s) (-150:\r)circle(\s) (0,0)circle(\s);
\path[inner sep=2pt] (-90:\r)node[below]{$\qz$} (90:\r/2)
node[right,inner sep=1pt]{$\qz$} (30:\r)node[anchor=-150]{$\qz''$};
\draw[->] (150:\r+0.4) to[curve through={(150:\r*0.8)..(-150:\r/2)..(-30:\r/2)}]
(30:\r/2);
\draw[->] (120:\r+0.4) ..controls (120:\r*0.8).. (105:\r/2);
\foreach \t in {0,60,-60,-120}
\draw[->] (\t:\r+0.5) -- (\t:\r+0.1);
\end{tikzpicture}
\caption{Isotopy toward the standard arcs.}
\label{fig-std-iso}
\end{figure}

Choose one of the standard arcs corresponding to $\qz''$. If the lantern is cut along
this arc as well as both standard arcs for $\qz$, the result is a disk. Therefore,
every diagram on the lantern can be drawn in the neighborhood of these arcs and the
boundary triangles. See Figure~\ref{fig-std-iso}. As explained in
\cite[Theorem~4.25]{GY}, we can draw 4 circles isotopic to the boundary triangles
such that the diagram consists of standard arcs in the smaller lantern bounded by these
curves. 
% An example is given in Figure~\ref{fig-std-iso}, where the thick lines
% are the chosen standard arcs, and the horizontal one is $\qz''$.
We can split along these 4 curves and apply counits, which does not change the
element in $\skeincr(\lantern)$. If we start in $\skeincrev(\lantern)$, then the
counits give coefficients in $\groundev$ by Lemma~\ref{lemma-counit-ev}, and the
standard arcs have even multiplicities by the even homology condition. Therefore, the
image in $\qglue(\lantern)$ is in the even part.
\end{proof}

%%%%%%%%%%%%%%%%%%%%%%%%%%%%%%%%%%%%%%%%%%%%%%%%%%%%%%%%%%%%%%%%%%%%%%%%%%%%
%%%%%%%%%%%%%%%%%%%%%%%%%%%%%%%%%%%%%%%%%%%%%%%%%%%%%%%%%%%%%%%%%%%%%%%%%%%%

\section{Change of triangulation}

We want to define maps from the skein module $\skein(M)$ that factor through the
quantum trace maps for suitable classes of triangulations. This requires an
understanding of the relations between quantum trace maps for triangulations related
by certain moves. In this paper, we consider the 3--2 and 2--0 moves for the
application of 3d-index.

\subsection{Compatibility with 3--2 moves}
\label{sec-qtr-32}

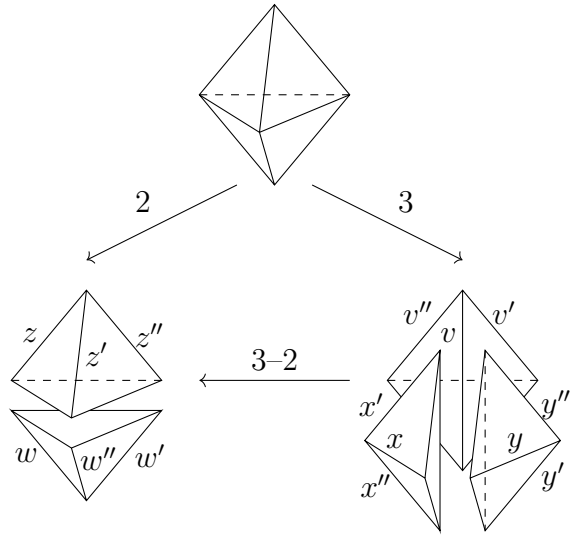
\begin{figure}[htpb!]
\centering
\begin{tikzpicture}
\path (-1,0)coordinate(3) (-0.2,-0.5)coordinate(1) (1,0)coordinate(2)
(0,1.2)coordinate(0) (0,-1.2)coordinate(4);
% \path (3) node[left]{3} (1) node[anchor=-120]{1} (2) node[right]{2} (0)
% node[above]{0} (4) node[below]{4};
\draw (0) -- (3) -- (4) -- (2) -- cycle (3) -- (1) -- (2) (0) -- (1) -- (4);
\draw[dashed] (3) -- (2);
\draw[->] (4) ++(-0.5,0) -- node[midway,above left]{2} +(-2,-1);
\draw[->] (4) ++(0.5,0) -- node[midway,above right]{3} +(2,-1);
\begin{scope}[xshift=-2.5cm,yshift=-3.8cm]
  \path (-1,-0.4)coordinate(3) (-0.2,-0.9)coordinate(1) (1,-0.4)
  coordinate(2) (0,-1.6)coordinate(4);
  \draw (4) -- node[midway,left]{$w$} (3) -- (1) -- (2) -- node[midway,right]{$w'$}
  cycle (4) -- node[pos=0.8,right,inner sep=2pt]{$w''$} (1) (3) -- (2);
  \path (-1,0)coordinate(3) (-0.2,-0.5)coordinate(1) (1,0)coordinate(2) (0,1.2)
  coordinate(0);
\fill[white] (1) -- (2) -- (3);
\draw (0) -- node[midway,left]{$z$} (3) -- (1) -- (2) -- node[midway,right]{$z''$}
cycle (0) -- node[midway,right,inner sep=2pt]{$z'$} (1);
\draw[dashed] (3) -- (2);
\end{scope}
\draw[<-] (-1,-3.8) -- node[midway,above]{3--2} (1,-3.8);
\begin{scope}[xshift=2.5cm,yshift=-3.8cm]
\path (-1,0)coordinate(3) (1,0)coordinate(2) (0,1.2)coordinate(0) (0,-1.2)coordinate(4);
\draw (0) -- node[pos=0.25,left]{$v''$} (3) -- (4) -- (2) -- node[pos=0.75,right]{$v'$}
cycle (0) -- node[pos=0.25,left,inner sep=2pt]{$v$} (4);
\draw[dashed] (3) -- (2);
\path (-1.3,-0.8)coordinate(3) (-0.5,-1.3)coordinate(1) (-0.3,0.4)
coordinate(0) (-0.3,-2) coordinate(4);
\fill[white] (3) -- (0) -- (4);
\draw (3) -- node[pos=0.4,left]{$x'$} (0) -- (4) -- (1) -- node[midway,above]{$x$}
cycle (3) -- node[pos=0.5,left]{$x''$} (4) (0) --(1);
\path (0.1,-1.3)coordinate(1) (1.3,-0.8)coordinate(2) (0.3,0.4)
coordinate(0) (0.3,-2) coordinate(4);
\fill[white] (4) -- (1) -- (0) -- (2);
\draw (0) -- (1) -- (4) -- node[pos=0.6,right]{$y'$} (2) -- node[pos=0.4,right]{$y''$}
cycle (1) -- node[midway,above,inner sep=2pt]{$y$} (2);
\draw[dashed] (0) -- (4);
\end{scope}
\end{tikzpicture}
\caption{3--2 Pachner move.}
\label{fig-32}
\end{figure}

Suppose $\triang_3\to\triang_2$ is a 3--2 move shown in Figure~\ref{fig-32}. Let
$\qtorus_2$ and $\qtorus_3$ be the quantum tori in the definition of
$\qglue(\triang_2)$ and $\qglue(\triang_3)$, respectively. Define
$\phi_{3,2}:\qtorus_2\to\qtorus_3$ by
\begin{equation}
\label{eq-32}
\begin{alignedat}{3}
\phi_{3,2}(\qz)&=\qv''\qx',&\qquad
\phi_{3,2}(\qz')&=\qx''\qy',&\qquad
\phi_{3,2}(\qz'')&=\qy''\qv',\\
\phi_{3,2}(\qw)&=\qx''\qv',&
\phi_{3,2}(\qw')&=\qv''\qy',&
\phi_{3,2}(\qw'')&=\qy''\qx',
\end{alignedat}
\end{equation}
and $\phi_{3,2}$ acts as identity on the remaining generators of $\qglue(\triang_2)$.
A straightforward calculation shows that \eqref{eq-32} has the correct $q$-commuting
relations, so this is a well-defined algebra map.

\begin{lemma}
\label{lemma-welldef-32}
The map above induces a well-defined map
\begin{equation}
\label{phi32}
\phi_{3,2}:\qglue(\triang_2)\to\qglue(\triang_3).
\end{equation}
\end{lemma}

Clearly, $\phi_{3,2}$ restricts to a map
$\qgluev(\triang_2)\to\qgluev(\triang_3)$.

\begin{proof}
It is easy to see from Figure~\ref{fig-32} that the edges $\qe_i\in\qtorus_2$ are
mapped to the corresponding monomials in $\qtorus_3$ by $\phi_{3,2}$. It remains to
check the Lagrangians are related by $\phi_{3,2}$. For the Lagrangians away from the
move region, the statement is trivial, and the cases of $\qz$ and $\qw$ are
symmetric. Here, we consider $\qz$.
\begin{equation}
\begin{split}
\phi_{3,2}(\qz^{-2}+\qz''^2-1)&=(\qv''\qx')^{-2}+(\qy''\qv')^2-1\\
&=\qv''^{-2}(1-\qx^2+L'_x)+\qv'^2(1-\qy^{-2}+L_y)-1\\
&=(-\qv''^{-2}\qv'^{-2}\qx^2\qy^2-1)\qv'^2\qy^{-2}+L''_v+\qv''^2L'_x+\qv'^2L_y\\
&=(q^{-1}\qv^2\qx^2\qy^2-1)\qv'^2\qy^{-2}+L''_v+\qv''^2L'_x+\qv'^2L_y.
\end{split}
\end{equation}
Here $L'_x=\qx'^{-2}+\qx^2-1$, $L_y=\qy^{-2}+\qy''^2-1$, and $L''_v=\qv''^{-2}+\qv'^2-1$
are Lagrangians in $\qtorus_3$. Note $\qv\qx\qy=\qe$ is the extra edge in
$\qtorus_3$, which commutes with the image of $\phi_{3,2}$. Then for any
$r\in\qtorus_2$,
\begin{equation}
\phi_{3,2}(r(\qz^{-2}+\qz''^2-1))=(q^{-1}\qe^2-1)\phi_{3,2}(r)\qv'^2\qy^{-2}
+(L''_v+\qv''^2L'_x+\qv'^2L_y).
\end{equation}
This shows that $\phi_{3,2}(\lideal{\qz^{-2}+\qz''^2-1})\subset
\rideal{\qe+q^{1/2}}+\lideal{\text{Lagrangian}}$.
\end{proof}

\begin{proposition}
\label{prop-qtr-32}
The quantum trace map is compatible with
$\phi_{3,2}$: $\qtr_{\triang_3}=\phi_{3,2}\circ\qtr_{\triang_2}$.
\end{proposition}

\begin{proof}
Let $\surface_i=\surface_{\triang_i}$ be the dual surfaces. Away from the move region
of Figure~\ref{fig-32}, the dual surfaces are identical. In the move region, the
surfaces are shown in Figure~\ref{fig-dual-32}. The left figure is $\surface_2$, and
the right figure is $\surface_3$. A few arcs on $\surface_3$ are omitted for clarity.
The dashed arc in the back of $\surface_2$ should also be on $\surface_3$, and the 3
new red arcs should go through the hole of $\surface_3$ to form complete circles.

To compare the quantum trace maps, we need to know how diagrams on $\surface_2$ is
isotoped onto $\surface_3$. From Figure~\ref{fig-dual-32}, it is clear that
$\surface_2$ is obtained from $\surface_3$ by a surgery along $B_e$. Given a diagram
on $\surface_2$, an isotopy makes it disjoint from the two disks bounded by the
copies of $B_e$, and reversing the surgery gives a diagram on $\surface_3$ that
defines the same link in $M$.

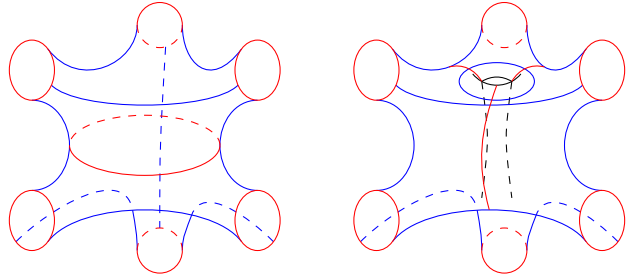
\begin{figure}[htpb!]
\centering
\begin{tikzpicture}[baseline=(ref.base)]
\begin{scope}[blue]
\draw[dashed] (0.2,1.6) ++(-75:0.3) ..controls (0.2,0).. (0.2,-1.1);
\draw (1.5,0.6) to[out=180,in=90] (1,0) to[out=-90,in=180] (1.5,-0.6)
(-1.5,0.6) to[out=0,in=90] (-1,0) to[out=-90,in=0] (-1.5,-0.6);
\draw ({1.5-0.3*cos(45)},{1+0.4*sin(45)})
..controls +({-0.3*sin(45)*2},{-0.4*cos(45)*2}) and +(0,-0.5).. (0.5,1.6); %top right
\draw ({-1.5+0.3*cos(45)},{1+0.4*sin(45)})
..controls +({0.3*sin(45)*2},{-0.4*cos(45)*2}) and +(0,-0.5).. (-0.1,1.6); %top left
\draw ({1.5-0.3*cos(30)},{1-0.4*sin(30)})
..controls +(-135:0.5) and +(-45:0.5).. ({-1.5+0.3*cos(30)},{1-0.4*sin(30)}); %top across
\draw ({1.5-0.3*cos(45)},{-1-0.4*sin(45)})
..controls +({-0.3*sin(45)*2},{0.4*cos(45)*2}) and +({0.3*sin(45)*2},{0.4*cos(45)*2})
.. ({-1.5+0.3*cos(45)},{-1-0.4*sin(45)})
node(A)[pos=0.55,coordinate]{} node(B)[pos=0.3,coordinate]{}; %bottom across
\draw (-0.1,-1.4) to[out=90,in=-80] (A) (0.5,-1.4) to[out=90,in=-110] (B);
\draw[dashed] ({-1.5-0.3*cos(45)},{-1-0.4*sin(45)}) to[in=100] (A);
\draw[dashed] ({1.5+0.3*cos(45)},{-1-0.4*sin(45)}) to[out=135,in=70] (B);
\end{scope}
\begin{scope}[red]
\foreach \x in {-1.5,1.5} \foreach \y in {-1,1}
\draw (\x,\y)circle[x radius=0.3,y radius=0.4];
\draw (0.5,1.6) arc[radius=0.3,start angle=0,end angle=180];
\draw[dashed] (-0.1,1.6) arc[radius=0.3,start angle=-180,end angle=0];
\draw[dashed] (0.5,-1.4) arc[radius=0.3,start angle=0,end angle=180];
\draw (-0.1,-1.4) arc[radius=0.3,start angle=-180,end angle=0];
\draw[dashed] (1,0) arc[x radius=1,y radius=0.4,start angle=0,end angle=180];
\draw (-1,0) arc[x radius=1,y radius=0.4,start angle=-180,end angle=0];
\end{scope}
\node (ref) at (0,0) {\phantom{$-$}};
\end{tikzpicture}
\qquad
\begin{tikzpicture}[baseline=(ref.base)]
\draw (-0.1,0.85) to[out=30,in=150] (0.3,0.85);
\draw (-0.22,0.95) to[curve through={(-0.1,0.85)..(0.1,0.8)..(0.3,0.85)}] (0.42,0.95);
\draw[dashed] (0.3,0.85) ..controls (0.2,0.1).. (0.3,-0.7) (-0.1,0.85)
..controls (0,0.1).. (-0.1,-0.7);
\begin{scope}[blue]
%\draw[dashed] (0.2,1.6) ++(-75:0.3) ..controls (0.2,0).. (0.2,-1.1);
\draw (0.1,0.85)circle[x radius=0.5,y radius=0.25];
\draw (1.5,0.6) to[out=180,in=90] (1,0) to[out=-90,in=180] (1.5,-0.6)
(-1.5,0.6) to[out=0,in=90] (-1,0) to[out=-90,in=0] (-1.5,-0.6);
\draw ({1.5-0.3*cos(45)},{1+0.4*sin(45)})
..controls +({-0.3*sin(45)*2},{-0.4*cos(45)*2}) and +(0,-0.5).. (0.5,1.6)
node(R)[midway,coordinate]{}; %top right
\draw ({-1.5+0.3*cos(45)},{1+0.4*sin(45)})
..controls +({0.3*sin(45)*2},{-0.4*cos(45)*2}) and +(0,-0.5).. (-0.1,1.6)
node(L)[midway,coordinate]{}; %top left
\draw ({1.5-0.3*cos(30)},{1-0.4*sin(30)})
..controls +(-135:0.5) and +(-45:0.5).. ({-1.5+0.3*cos(30)},{1-0.4*sin(30)}); %top across
\draw ({1.5-0.3*cos(45)},{-1-0.4*sin(45)})
..controls +({-0.3*sin(45)*2},{0.4*cos(45)*2}) and +({0.3*sin(45)*2},{0.4*cos(45)*2})
.. ({-1.5+0.3*cos(45)},{-1-0.4*sin(45)})
node(F)[midway,coordinate]{}
node(A)[pos=0.55,coordinate]{} node(B)[pos=0.3,coordinate]{}; %bottom across
\draw (-0.1,-1.4) to[out=90,in=-80] (A) (0.5,-1.4) to[out=90,in=-110] (B);
\draw[dashed] ({-1.5-0.3*cos(45)},{-1-0.4*sin(45)}) to[in=100] (A);
\draw[dashed] ({1.5+0.3*cos(45)},{-1-0.4*sin(45)}) to[out=135,in=70] (B);
\end{scope}
\begin{scope}[red]
\foreach \x in {-1.5,1.5} \foreach \y in {-1,1}
\draw (\x,\y)circle[x radius=0.3,y radius=0.4];
\draw (0.5,1.6) arc[radius=0.3,start angle=0,end angle=180];
\draw[dashed] (-0.1,1.6) arc[radius=0.3,start angle=-180,end angle=0];
\draw[dashed] (0.5,-1.4) arc[radius=0.3,start angle=0,end angle=180];
\draw (-0.1,-1.4) arc[radius=0.3,start angle=-180,end angle=0];
\draw (0.1,0.8) to[out=-110,in=105] (F) (0.3,0.85) to[out=60,in=165] (R) (-0.1,0.85)
to[out=120,in=15] (L);
%\draw[dashed] (0.1,0.8) to[out=-80,in=75] (F);
\end{scope}
\node (ref) at (0,0) {\phantom{$-$}};
\end{tikzpicture}
\caption{Dual surfaces $\surface_2$ and $\surface_3$ (a few arcs omitted for clarity).}
\label{fig-dual-32}
\end{figure}

Now start with a diagram on $\surface_2$. By an isotopy, we can assume that the
diagram is in a small neighborhood of the standard arcs and the $A$-circles except
the middle one in Figure~\ref{fig-dual-32}. This is possible since the complement of
this neighborhood is disks. We can further assume that near the standard arcs, the
diagram is just parallel strands of standard arcs.

First assume that the 6 ``boundary'' faces of the move region are not paired with
each other. The quantum trace requires splitting along $A$-circles, which creates a
sum over states. The sums outside the move region are identical for $\triang_2$ and
$\triang_3$ by assumption, and the sums inside the move region has finitely many
cases with 9 types of standard arcs and 4 state combinations each. Then it is
straightforward to check that they are given by $\phi_{3,2}$. Here, we only give one
calculation as an example. Consider the arc in the back in Figure~\ref{fig-dual-32}
with $-$ states at both endpoints. On $\surface_2$, it is split into two standard
arcs. The new endpoints must have $-$ states for the diagram to be nonzero in
$\qglue(\triang_2)$, in which case it evaluates to $\qz'\qw''$. On $\surface_3$, it
is a single standard arc $\qv$. They are related by $\phi_{3,2}$ since
\begin{equation}
  \phi_{3,2}(\qz'\qw'')=(\qx''\qy')(\qy''\qx')=(\iunit q^{1/4})^2(\qx\qy)^{-1}
  =(-q^{1/2}\qe^{-1})\qv=\qv.
\end{equation}
As in the proof of Lemma~\ref{lemma-welldef-32}, $\qe=\qx\qy\qv$ can be set to
$-q^{1/2}$ since it commutes with the image of $\phi_{3,2}$. The other arcs are
similar.

If there are extra face pairings in the move region, we can double the corresponding
$A$-curves to create a buffer zone, which can be absorbed using the counit of the
annulus. The remainder of the argument is the same.
\end{proof}

\subsection{Compatibility with 2--0 moves}
\label{sec-qtr-20}

\begin{figure}[htpb!]
\centering
\begin{tikzpicture}
  \path (-1,0)coordinate(3) (1,0)coordinate(2) (0,1.2)coordinate(0) (0,-1.2)
  coordinate(4);
%\path (3) node[left]{3} (2) node[right]{2} (0) node[above]{0} (4) node[below]{4};
\draw (3) -- (4) -- (2) -- (0) -- cycle -- (2);
\draw[<-] (1.5,0) -- node[midway,above]{2--0} (3,0);
\begin{scope}[xshift=4.5cm]
  \path (-1,0)coordinate(3) (1,0)coordinate(2) (0,1.2)coordinate(0) (0,-1.2)
  coordinate(4);
\draw (3) -- (4) -- (2) -- (0) -- cycle (3) to[out=-15,in=-165] (2);
\draw[dashed] (3) to[out=15,in=165] (2) (0) -- (4);
\end{scope}
\draw[->] (6,-0.5) -- node[midway,above,sloped]{front} +(1.5,-0.5);
\draw[->] (6,0.5) -- node[midway,above,sloped]{back} +(1.5,0.5);
\begin{scope}[xshift=9cm,yshift=-1.5cm]
  \path (-1,0)coordinate(3) (1,0)coordinate(2) (0,1.2)coordinate(0) (0,-1.2)
  coordinate(4);
\draw (3) -- (4) -- node[midway,below right,inner sep=0pt]{$z'$} (2) --
node[midway,above right,inner sep=2pt]{$z''$} (0) -- cycle (3) to[out=-15,in=-165] (2);
\draw[dashed] (0) -- (4);
\path (0,0)node[right]{$z$};
\end{scope}
\begin{scope}[xshift=9cm,yshift=1.5cm]
  \path (-1,0)coordinate(3) (1,0)coordinate(2) (0,1.2)coordinate(0) (0,-1.2)
  coordinate(4);
\draw (3) -- (4) -- (2) -- node[midway,above right,inner sep=2pt]{$w'$} (0) --
node[midway,above left,inner sep=0pt]{$w''$} cycle (0) -- (4);
\draw[dashed] (3) to[out=15,in=165] (2);
\path (0,0)node[right]{$w$};
\end{scope}
\end{tikzpicture}
\caption{2--0 move.}
\label{fig-20}
\end{figure}
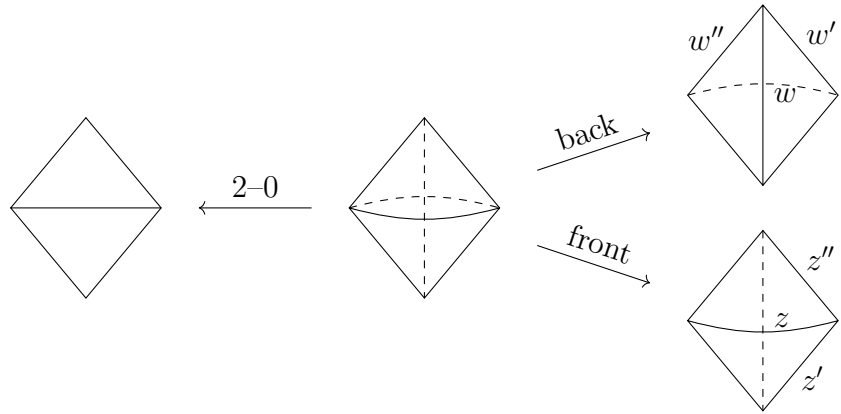

Suppose $\triang_2\to\triang_0$ is a 2--0 move. See Figure~\ref{fig-20}. The
construction is similar to the 3--2 move, so less details will be provided.

\begin{remark}\label{rem-20}
Unlike the 3--2 move, the construction below only works for the even part. This can
be traced to the extra quotient mentioned in Remark~\ref{rem-extra-gen}. If the full
quantum trace is used, then the 2--0 move will be fully compatible, but a lot more
technical details are required. We choose to omit the complete picture to focus on
the 3d-index, which only requires the even part.
\end{remark}

Suppose $M$ has no non-peripheral $\ints/2$-homology, so
presentation~\eqref{eq-qgluev} is valid. The quantum gluing modules are related by
a map
\begin{equation}
\label{phi20}
\phi_{2,0}:\qgluev(\triang_0)\to\qgluev(\triang_2)
\end{equation}
that inserts a tensor factor of $1$ in the slots of the new tetrahedra.

\begin{lemma}
\label{lemma-welldef-20}
$\phi_{2,0}$ is well-defined.
\end{lemma}

\begin{proof}
The description above certainly defines a map between the quantum tori. All
Lagrangians and most edges of $\triang_0$ are identical to those of $\triang_2$. The
only nontrivial check is the edge $\qE$ shared by the two faces on the $\triang_0$
side, which is split into two edges $\qE',\qE''$ on the $\triang_2$ side. By
inspection, we get $\phi_{2,0}(\qE)=\qE'\qE''\qE_0^{-1}$, where $\qE_0=\qZ\qW$ is the
new vertical edge in Figure~\ref{fig-20}. Thus, $\phi_{3,2}(\rideal{\qE-q})$ is still
in $\rideal{\qE_i-q}$.
\end{proof}

\begin{proposition}
\label{prop-qtr-20}
Suppose $M$ has no non-peripheral $\ints/2$-homology. Then the quantum trace map is
compatible with $\phi_{2,0}$ on the even part: $\qtr_{\triang_2}=\phi_{2,0}
\circ\qtr_{\triang_0}$.
\end{proposition}

\begin{proof}
Let $\surface_i=\surface_{\triang_i}$ be the dual surfaces. To go from $\surface_0$
to $\surface_2$, we remove a small neighborhood of the two $A$-circles in the move
region and glue in the surface in Figure~\ref{fig-dual-20}. Note two of the standard
arcs in each removed annuli become standard arcs in the inserted surface, which are
the horizontal ones in Figure~\ref{fig-dual-20}.

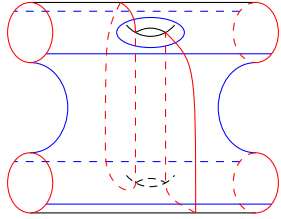
\begin{figure}[htpb!]
\centering
\begin{tikzpicture}[baseline=(ref.base)]
\draw (-1.5,1.4) -- (1.5,1.4) (-1.5,-1.4) -- (1.5,-1.4);
\draw (-0.1,1) to[out=30,in=150] (0.3,1);
\draw (-0.22,1.1) to[curve through={(-0.1,1)..(0.3,1)}] (0.42,1.1);
\draw[densely dashed] (-0.1,-1) to[out=30,in=150] (0.3,-1);
\draw[densely dashed] (-0.22,-0.9) to[curve through={(-0.1,-1)..(0.3,-1)}] (0.42,-0.9);
\begin{scope}[blue]
\draw (0.1,1)circle[x radius=0.45,y radius=0.2];
\draw (1.5,0.6) to[out=180,in=90] (1,0) to[out=-90,in=180] (1.5,-0.6)
(-1.5,0.6) to[out=0,in=90] (-1,0) to[out=-90,in=0] (-1.5,-0.6);
\draw ({1.5+0.3*cos(45)},{1-0.4*sin(45)}) -- ({-1.5+0.3*cos(45)},{1-0.4*sin(45)}); %top
\draw[dashed] ({1.5-0.3*cos(45)},{1+0.4*sin(45)})
-- ({-1.5-0.3*cos(45)},{1+0.4*sin(45)}); %top
\draw ({1.5+0.3*cos(45)},{-1-0.4*sin(45)})
-- ({-1.5+0.3*cos(45)},{-1-0.4*sin(45)}); %bottom
\draw[dashed] ({1.5-0.3*cos(45)},{-1+0.4*sin(45)})
-- ({-1.5-0.3*cos(45)},{-1+0.4*sin(45)}); %bottom
\end{scope}
\begin{scope}[red]
\foreach \y in {-1,1} {
  \draw (-1.5,\y)circle[x radius=0.3,y radius=0.4] (1.5,\y-0.4)
  arc[x radius=0.3,y radius=0.4,start angle=-90,end angle=90];
\draw[dashed] (1.5,\y+0.4)arc[x radius=0.3,y radius=0.4,start angle=90,end angle=270];
}
\draw (0.3,1) ..controls (0.7,0.6).. (0.7,-1.4);
\draw[dashed] (0.7,-1.4) to[out=150,in=-90] (0.3,-1) -- (0.3,1);
\draw (-0.1,1) to[out=90,in=-30] (-0.3,1.4);
\draw[dashed] (-0.3,1.4) ..controls (-0.5,1).. (-0.5,-0.6)
..controls +(0,-0.5) and +(0,-0.2).. (-0.1,-1) -- (-0.1,1);
\end{scope}
\node (ref) at (0,0) {\phantom{$-$}};
\end{tikzpicture}
\caption{Surface inserted into $\surface_0$}\label{fig-dual-20}
\end{figure}

In the calculation of $\qtr_{\triang_0}$, we can double the $A$-circles to make the
annuli in the move region. We can also isotope any diagram to be standard arcs in the
annuli that remain standard after the move. This gives 4 types of standard arcs (2 in
each annuli). They evaluate to $0$ or $1$ by the counit depending on the states.

On the other hand, for $\qtr_{\triang_2}$, these standard arcs evaluate to $0$ or
$(\qZ'\qW'')^{\pm1}=(\qZ''\qW')^{\mp1}$. Again, $\qE_0=\qZ\qW$ is set to $q$ in this
statement. Now we reduce such a monomial using Lagrangians $L_z=\qZ^{-1}+\qZ''-1$
and $L''_w=\qW''^{-1}+\qW'-1$.
\begin{equation}
\begin{split}
\qZ''\qW'&=(1-\qZ^{-1}+L_z)\qW'=(1-\qE_0^{-1}\qW)\qW'+\qW'L_z\\
&=\qW'+q\qE_0^{-1}\qW''^{-1}+\qW'L_z\\
&=1+L''_w+\qW'L_z+(q\qE_0^{-1}-1)\qW''^{-1}.
\end{split}
\end{equation}
Although $q\qE_0^{-1}-1$ does not commute with every term, it commutes with the whole
$\qZ''\qW'$. Therefore, $(\qZ''\qW')^n-1\in\rideal{\qE_0-q}+\lideal{L_z,L''_w}$ for
$n\ge0$. The same is true for $n<0$ using an analogous calculation with $\qZ'\qW''$.
This matches $\phi_{2,0}\circ\qtr_{\triang_0}$.
\end{proof}

%%%%%%%%%%%%%%%%%%%%%%%%%%%%%%%%%%%%%%%%%%%%%%%%%%%%%%%%%%%%%%%%%%%%%%%%%%%% 
%%%%%%%%%%%%%%%%%%%%%%%%%%%%%%%%%%%%%%%%%%%%%%%%%%%%%%%%%%%%%%%%%%%%%%%%%%%%

\section{The 3d-index}
\label{sec.3dindex}

In this section we review the 3d-index of~\cite{DGG1} and its relation to the normal
surfaces of an ideal triangulation~\cite{GHHR}, following the notation of i.b.i.d. \red{what is ibid?}

As usual, we let $\indexRp$ (resp., $\indexR$) denote the rings of formal power series 
(resp., formal Laurent series) in a variable $q^{1/2}$ with integer coefficients.
Note $\indexR$ is naturally a $\groundev$-module.

\subsection{Tetrahedron index}

The DGG 3d-index is a sum over a lattice of products of building blocks, one per
tetrahedron. The building block is the \term{tetrahedron index}
$I_\Delta(m,e)\in \indexR$ of~\cite{DGG1} defined by

\begin{equation}
\label{ID}
I_\Delta(m,e)=\sum_{n=e_-}^\infty
(-1)^n\frac{q^{\frac{1}{2}n(n+1)-\left(n+\frac{1}{2}e\right)m}}{(q;q)_n(q;q)_{n+e}},
\end{equation}
where $e_-=\max\{-e,0\}$ and $(q;q)_n=\prod_{i=1}^n(1-q^i)$ is the quantum factorial
(also known as $q$-Pochhammer symbol).

The tetrahedron index satisfies some symmetries that can be conveniently expressed
by saying that the 3-variable function 
\begin{equation}
J_\Delta(a,b,c)=(-q^{1/2})^{-b}I_\Delta(b-c,a-b)\in\indexR
\end{equation}
is invariant under all six permutations of $a,b,c$, as was found out in
~\cite[Sec.4.7]{GHRS}. These symmetries are reminiscent of the orbit of a shape
of tetrahedron under all permutations (orientation preserving, or not).

In addition, the tetrahedron index satisfies a translation
invariance, a linear $q$-difference equation, a pentagon identity and a
quadratic identity:
\begin{subequations}
\begin{align}
\label{eq-J-diff}
J_\Delta(a-s,b-s,c-s) &=(-q^{1/2})^s J_\Delta(a,b,c) ,\\
\label{eq-J-Lagr}
q^{\frac{1}{2}(a-b)}J_\Delta(a,b,c-1)+q^{\frac{1}{2}(c-b)}J_\Delta(a+1,b,c)
&=J_\Delta(a,b,c),\\
\label{eq-J-penta}
%\text{pentagon}
\sum_{k\in\ints}q^kJ_\Delta(k,a+f,b+d)J_\Delta(k,a+e,c+d)J_\Delta(k,b+e,c+f)
&=J_\Delta(a,b,c)J_\Delta(d,e,f),\\
\label{eq-J-quad}
%\text{quadratic}
\sum_{k\in\ints}q^kJ_\Delta(a+k,c,d)J_\Delta(b+k,c,d)&=q^{-a}\delta_{a,b}.
\end{align}
\end{subequations}

\begin{remark}
\label{rem.qnot1}
It is easy to see that the tetrahedron index~\eqref{ID} is a series convergent
on the unit disk $\abs{q}<1$. But more is true. Replacing $q$ by $q^{-1}$ in the
$q$-hypergeometric sum, it follows that the tetrahedron index satisfies a duality
\begin{equation}
\label{eq-index-duality}
I_\Delta(m,e)(q^{-1})=I_\Delta(-m,-e)(q),\qquad\text{or}\quad
J_\Delta(a,b,c)(q^{-1})=J_\Delta(-a,-b,-c)(q).
\end{equation}
It follows that the tetrahedron index $I_\Delta$ as well as the symmetric form
$J_\Delta$ are $q$-hypergeometric series that are convergent for $\abs{q} \neq 1$.

Define the involution $\iota$ on the algebra of holomorphic functions of
$q^{1/2}\in\cx^\times\setminus S^1$ by
\begin{equation}
\label{iota}
\iota(q^{1/2}) = q^{-1/2} \,.
\end{equation}
Then we can write $\iota J_\Delta(a,b,c)=J_\Delta(-a,-b,-c)$.
% By abuse of notations, we write $\iota:\indexR\to\indexR$.
\end{remark}

\subsection{Basics on normal surfaces}
\label{sub.Idef}

%We now fix an oriented 3-manifold $M$ with $r$ torus boundary components, and an
%1-efficient triangulation $\triang$ of $M$ with $N$ tetrahedra.

As we will see in the later section, there is a close connection between the 3d-index
in the next section and normal surface theory. 
In this section we recall some basics on normal surface theory following the
definitions from Sections 6--8 of~\cite{GHHR} where the reader can find the
references for the results stated below.

Normal surfaces were introduced by Haken and Kneser as a convenient way to deduce
decision problems in 3-dimensional topology to piece-wise linear statements. 
Throughout this section, we fix a connected, oriented 3-manifold whose boundary
consists of $r \geq 1$ torii, and an ideal triangulation $\triang$ of $M$ with $N$
tetrahedra. 
A normal surface is $S$ in $M$ is a surface that intersects each tetrahedron in
triangles or quadrilaterals. The number of triangles (there are four per tetrahedron)
and quadrilaterals (three per tetrahedron) defines a vector in
$\BN^{7N} \subset \BR^{7N}$ that satisfies some linear matching equations
(one per face of $\triang$), where $\BN=\{0,1,2,\ldots\}$ is the set of nonnegative
integers. Casson, Rubinstein and Tollefson noticed that
normal surfaces can be uniquely reconstructed (up to multiples of peripheral surfaces)
by their quadrilateral coordinates, and below we will be using only the quadrilateral
coordinates
\begin{equation}
\label{Sabc}
S=(a_1, b_1, c_1, \ldots , a_N, b_N, c_N) \in \BR^{3N}
\end{equation}
of a normal surface $S$. These coordinates satisfy a system of linear $Q$-matching
equations, one per edge of $\triang$. Since there are $N$ edges, the
$N \times 3N$ matrix of the $Q$-matching equations is given by $(A|B|C)D$ where
$(A|B|C)$ is the $N \times 3N$ matrix of gluing equations of $\triang$ with rows
indexed by the edges of $\triang$, and $D$ is the block $3N \times 3N$ diagonal
matrix of $N$ copies of 
\begin{equation}
\label{Dmat}
\begin{pmatrix} 0 & 1 & -1 \\ -1 & 0 & 1 \\ 1 & -1 & 0\end{pmatrix}
\end{equation}

The vector space $\BR^{3N}$ has a skew-symmetric pairing
$\omega$ on $\BR^{3N}$ defined as follows
\begin{equation}
\label{omegadef}
\omega: \BR^{3N} \times \BR^{3N} \to \BR, \qquad
\omega(x,x') = x^t D x' = \sum_{j=1}^N \left( 
\begin{vmatrix} a_j & a_j' \\ b_j & b_j' \end{vmatrix} +  
\begin{vmatrix} b_j& b_j' \\ c_j & c_j' \end{vmatrix} +   
\begin{vmatrix} c_j& c_j' \\ a_j & a_j' \end{vmatrix} \right) 
\end{equation}
for
\begin{equation}
\label{xx'}
x=(a_1, b_1, c_1, \ldots , a_N, b_N, c_N),
x'=(a'_1, b'_1, c'_1, \ldots , a'_N, b'_N, c'_N) \in   (\BR^3)^N \,.
\end{equation}
The $Q$-matching equations can be expressed in terms of the skew-symmetric pairing
as follows. $S \in \BR^{3N}$ satisfies the $Q$-matching equations if and only
if
\begin{equation}
\label{oEiS}
\omega(E_i,S) = 0, \qquad i=1,\dots,n
\end{equation}
where $E_i$ is the $i$th row of $(A|B|C)$. 

We denote by $Q(\triang;\BR)$ the real $Q$-normal surface
solution space. It is a real vector space that contains a lattice $Q(\triang;\BZ)$
that consists of all $\BZ$-solutions to the matching equations. There are three
types of vectors in $Q(\triang;\BZ)$, namely
\begin{itemize}
\item
  edge solutions $E_i$, for $i=1,\dots,N$ whose coordinates are the
  number of tetrahedra around the edge $E_i$ for each quad type. This vector is
  exactly the exponent vector of the gluing (also known as Neumann--Zagier) equations
  of $\triang$. 
\item
  tetrahedra solutions $\Delta_j$ for $j=1,\dots,N$ whose coordinates consist
  of $1,1,1$ in the $j$-th spots and $0$ in all other spots.
\item
  peripheral solutions $M_k$ and $L_k$ for $k=1,\dots,r$ whose coordinates are
  the exponent vectors of the cusp gluing equations of $\triang$, with a fixed
  basis for the integer homology of each boundary torus. 
\end{itemize}

While it is clear that the tetrahedra solutions satisfy the matching equations, the
fact that the edge and the peripheral solutions satisfy the matching equations is
equivalent to the symplectic property of the gluing equations proven by Neumann--Zagier,
which asserts that all symplectic products of $E_i,M_k,L_k$ for $i=1,\dots,N$
and $k=1,\dots,r$ are zero except that $\omega(L_k,M_k)=2=-\omega(M_k,L_k)$.
Not all of the edge solutions are independent in $Q(\triang;\BZ)$, but one can always
find a suitable subset of $N-r$ independent edge solutions. 

Kang--Rubinstein proved that the lattice $Q(\triang;\BZ)$ has rank $2N+r$, with
a basis that consists of $N-r$ edge solutions, $N$ tetrahedra solutions and $2r$ 
peripheral solutions. On this lattice and its corresponding real vector space
there is a linear function (an analogue of the Euler characteristic of a surface)
\begin{equation}
\label{Sxy}
\chi: Q(\triang;\BR) \to \BR, \qquad \chi(S)=\sum_i -2x_i- \sum_j y_j
\end{equation}
defined for 
\begin{equation}
S=\sum_i x_i E_i + \sum_j y_j \Delta_j + \sum_k (p_k M_k + q_k L_k) \,.
\end{equation}

The vector space $\BR^{3N}$ has a quadratic form, the double-arc function, defined
as follows
\begin{equation}
\label{deltadef}
\delta: \BR^{3N} \to \BR, \qquad \delta(x) = \frac{1}{2} x^t D^\sym x
= \sum_{j=1}^N (a_j b_j + b_j c_j + c_j a_j)
\end{equation}
where $D^\sym$ is the block $3N \times 3N$ diagonal matrix of $N$ copies of
\begin{equation}
\label{Dmatsym}
\begin{pmatrix} 0 & 1 & 1 \\ 1 & 0 & 1 \\ 1 & 1 & 0\end{pmatrix}
\end{equation}
and $x$ is as in~\eqref{xx'}. Its corresponding bilinear form is given by
\begin{equation}
\label{delta2def}
\delta: \BR^{3N} \times \BR^{3N} \to \BR, \qquad
\delta(x,x') = \frac{1}{2} x^t D^\sym x' \,.
%= \sum_{j=1}^N (a_j b_j + b_j c_j + c_j a_j)
\end{equation}
The double-arc function determines the $q^{1/2}$-degree of the tetrahedron index
in its symmetric version $J_\Delta(a,b,c)$; see~\cite[Eqn.(8)]{GHHR}
\begin{equation}
\label{degJ}
\deg J_\Delta(a,b,c) = %\frac{1}{2} ()
\delta(a-m,b-m,c-m)-m, \qquad m=\min\{a,b,c\}
\,.
\end{equation}
$\delta$ satisfies a positivity property: for all $x, x' \in \BR^{3N}_+$ (where
$\BR_+ =[0,\infty)$) we have
\begin{equation}
\label{dsuper}
\delta(x+x') \geq \delta(x) + \delta(x'),
\end{equation}
which is a consequence of the fact that $D$ has nonnegative entries. 

The lattice $Q(\triang;\BZ)$ contains a sublattice $Q_0(\triang;\BZ)$ of solutions of
$Q$-matching equations corresponding to closed surfaces of rank $2N-r$ with basis
$N-r$ edge solutions generating a sublattice $\BE$ and $N$ tetrahedra solutions
generating a sublattice $\BD$. It is this lattice that will play a role in the
3d-index through the bijection
\begin{equation}\label{eq-biject}
\BZ^{N-r} \simeq Q_0(\triang;\BZ)/\BD =(\BE + \BD)/\BD,
\qquad k=(k_1,\dots,k_{N-r}) \mapsto S(k)=\sum_{i=1}^{N-r} k_i E_i \,. 
\end{equation}

\subsection{Definition of the 3d-index}
\label{sub.convergence}

We now have all the ingredients to define the 3d-index. We fix an 1-efficient
triangulation $\triang$ with $N$ tetrahedra on an oriented 3-manifold with $r$ torus
boundary components.

Let
$\qtorusev_\triang:=\bigotimes_{j=1}^N\qtorusev\langle\qZ_j,\qZ_j',\qZ''_j\rangle$.
It has the Weyl-ordered basis $\{[\Zar^S]\mid S\in\ints^{3N}\}$ where
\begin{equation}
\label{SabcZ}
\Zar^S = \prod_{j=1}^N \qZ_j^{a_j} (\qZ_j')^{b_j} (\qZ_j'')^{c_j}, \qquad
S=(a_1, b_1, c_1, \ldots , a_N, b_N, c_N) \in \BZ^{3N}
\end{equation}
Note $[\Zar^{E_i}]=\qE_i$, and $[\Zar^{\Delta_j}]=[\qZ_j\qZ'_j\qZ''_j]$. Note also
that the $q$-commuting relations \eqref{eq-zzp} (after squaring) matches the
definition \eqref{omegadef}.

\begin{definition}
\label{def.Idef}
Consider the unique $\BZ[q^{\pm 1/2}]$-linear map 
\begin{equation}
\label{It1}
I_\triang: \qtorusev_\triang
=\bigotimes_{j=1}^N\qtorusev\langle\qZ_j,\qZ_j',\qZ''_j\rangle \to
\indexR
\end{equation}
given by the following series in $\indexR$
\begin{equation}
\label{3dIdef}
I_\triang([\Zar^{S_0}]) =
\sum_{k \in \BZ^{N-r}}
(-q^{\frac{1}{2}})^{-\chi(S(k))} q^{\frac{1}{2} \omega(S_0, S(k))}
J\left(-S_0 + S(k)\right),
\end{equation}
where $J$ is the extension of $J_\Delta$ from $\BZ^3$ to $\BZ^{3N}$ by
\begin{equation}
J: \BZ^{3N} \to \indexR, \qquad 
S=(a_1, b_1, c_1, \ldots , a_N, b_N, c_N) \mapsto 
\prod_{j=1}^N J_\Delta(a_j,b_j,c_j) \,.
\end{equation}
\end{definition}

\begin{theorem}
\label{thm.3d}
$I_\triang$ is well-defined (i.e., its image are convergent power series in $\indexR$).

If $M$ has no non-peripheral $\ints/2$-homology, then $I_\triang$ descends to a map
\begin{equation}
\label{It2}
I_\triang: \qgluev(\triang) \to \indexR.
\end{equation}
\end{theorem}

\begin{proof}
Using the bijection \eqref{eq-biject}, we can rewrite the definition as
\begin{equation}\label{3dIdef2}
I_\triang([\Zar^{S_0}]) =\sum_{[S] \in Q_0(\triang;\BZ)/\BD}
(-q^{\frac{1}{2}})^{-\chi(S)} q^{\frac{1}{2} \omega(S_0,S)} J(-S_0 + S).
\end{equation}
Note the summand does not depend on the representative $S$, since
$(-q^{\frac{1}{2}})^{-\chi(S)}J(-S_0 + S)$ is well-defined by \cite[Section~8]{GHHR},
and $\omega(\ast,\Delta_j)=0$ by direct calculation.

First, we prove the convergence of our extended 3d-index for $S_0 \in \ints^{3N}$.
The case $S_0\in Q_0(\triang;\ints)$ is covered by \cite[Theorem~8.6]{GHHR}, so we
assume $S_0\notin Q_0(\triang;\ints)$. In this case, we follow the proof of
\cite[Theorem~8.8]{GHHR}.

Let
\begin{equation}
Q_{S_0}(\triang;\ints)=\ints S_0\oplus Q_0(\triang;\ints),\qquad
Q^1_{S_0}(\triang;\ints)=(-S_0)+Q_0(\triang;\ints)\subset Q_{S_0}(\triang;\ints).
\end{equation}
Define $\chi_{S_0}(kS_0+S')=\chi(S')$ for any $k\in\ints$ and $S'\in
Q_0(\triang;\ints)$. Then the definition \eqref{3dIdef2} can be rewritten once again as
\begin{equation}\label{3dIdef3}
I_\triang([\Zar^{S_0}]) =\sum_{[S] \in Q^1_{S_0}(\triang;\BZ)/\BD}
(-q^{\frac{1}{2}})^{-\chi_{S_0}(S)} q^{\frac{1}{2} \omega(S_0,S)} J(S).
\end{equation}
We can choose $\BD$-coset representatives which are convenient for the degree formula
\eqref{degJ}. Given $S\in\ints^{3N}$, there is a unique translation $S^\ast$ by an
element of $\BD$ such that $\min\{a_j,b_j,c_j\}=0$ for every $j=1,\dots,N$. (If $S\in
Q(\triang;\ints)$, then the corresponding normal surface has at most two types of
quads at each tetrahedron). This defines an injective map
$\ints^{3N}/\BD\to\nats^{3N}$. Then the sum in \eqref{3dIdef3} can be replaced by
$S^\ast\in Q^{1\ast}_{S_0}(\triang;\nats):=(Q^1_{S_0}(\triang;\BZ)/\BD)^\ast$, and
the degree of the summand is
\begin{equation}
d(S^\ast)=-\chi_{S_0}(S^\ast)+\omega(S_0,S^\ast)+\delta(S^\ast).
\end{equation}

Let
$Q^\ast_{S_0}(\triang;\nats)=\big(\nats(-S_0)+Q^\ast_0(\triang;\ints)\big)\cap\nats^{3N}$.
This is the intersection of a lattice with a rational polyhedral cone, hence every
vector in $Q^\ast_{S_0}(\triang;\nats)$ is an $\nats$-linear combination of a finite set
$F_i$ of fundamental solutions. % for $i=1,\dots,m$.
Let $I_0$ and $I_1$ denote the set of $i$ that correspond to fundamental solutions
$F_i$ in $Q_0(\triang;\nats)$ and $Q^{1\ast}_{S_0}(\triang;\nats)$ respectively. Then
for $S^\ast\in Q^{1\ast}_{S_0}(\triang;\nats)$, we can write
\begin{equation}
S^\ast=F_k+\sum_{i\in I_0}x_iF_i\qquad
\text{for some $k\in I_1$ and $x_i\in\nats$}.
\end{equation}

The function $d$ is defined for all of $Q^\ast_{S_0}(\triang;\nats)$. By
\eqref{dsuper} and the linearity of the first two terms, $d$ is superadditive, so for
$S^\ast$ as above,
\begin{equation}
d(S^\ast)\ge d(F_k)+\sum_{i\in I_0}d(x_iF_i).
\end{equation}
For $i\in I_0$, consider the term
\begin{equation}\label{fi}
d(x_i F_i) = -\chi(F_i) x_i + \delta(F_i) x_i^2, \qquad x_i \geq 0 \,.
\end{equation}
Since $\triang$ is 1-efficient, for each $i$, we have either $\delta(F_i)=0$ and
$-\chi(F_i) \geq 1$ or $\delta(F_i) \geq 1$. Therefore, $d(x_iF_i) \to \infty$ as
$x_i \to \infty$. Since there are finitely many choices of $F_k,k\in I_1$, and
$d(S^\ast)\to\infty$ if any $x_i\to\infty$, we see $d(S^\ast) \leq D$ has finitely
many solutions $S^\ast$. This implies the convergence of \eqref{3dIdef3}.

Now we show that $I_\triang$ is compatible with the quotient
$\qtorusev_\triang\onto\qgluev(\triang)$. By the homology condition,
$\qgluev(\triang)$ is given by the presentation \eqref{eq-qgluev} (with $\qZ'_j$
restored). This means we need to check
\begin{enumerate}[(a)]
\item $I_\triang(\qE_i[\Zar^{S_0}])=qI_\triang([\Zar^{S_0}])$.
\item $I_\triang([\qZ_j\qZ'_j\qZ''_j][\Zar^{S_0}])=-q^{\frac{1}{2}}I_\triang([\Zar^{S_0}])$.
\item $I_\triang([\Zar^{S_0}](\qZ_j^{-1}+\qZ''_j-1))=0$.
\end{enumerate}

For (a), as remarked after the definition \eqref{eq-qglue-def} of $\qglue(\triang)$,
we only need to consider $i=1,\dotsc,N-r$. Since $\qE_i [\Zar^{S_0}] =
q^{\frac{1}{2}\omega(E_i,S_0)} [\Zar^{S_0+E_i}]$, it follows that
\begin{equation}
\begin{split}
I_\triang(\hat E_i [\Zar^{S_0}]) &= q^{\frac{1}{2}\omega(E_i,S_0)}
I_\triang([\Zar^{S_0+E_i}]) \\
&=q^{\frac{1}{2}\omega(E_i,S_0)}
\sum_{[S] \in Q_0(\triang;\BZ)/\BD}
(-q^{\frac{1}{2}})^{-\chi(S)} q^{\frac{1}{2} \omega(S_0+E_i,S)} J(-S_0-E_i + S) \,.
\end{split}
\end{equation}
Since $E_i \in Q_0(\triang;\BZ)$, shifting $S \mapsto S+E_i$, we obtain that
\begin{equation}
I_\triang(\hat E_i [\Zar^{S_0}]) = q^{\frac{1}{2}\omega(E_i,S_0)}
\sum_{[S] \in Q_0(\triang;\BZ)/\BD}
(-q^{\frac{1}{2}})^{-\chi(S+E_i)} q^{\frac{1}{2} \omega(S_0+E_i,S+E_i)} J(-S_0 + S).
\end{equation}
Definition~\eqref{Sxy} implies that $\chi(E_i)=-2$ and \eqref{oEiS} implies that
$\omega(E_i,S)=0$. This gives that $\omega(S_0+E_i,S+E_i)=\omega(S_0,S)+\omega(S_0,E_i)$.
Together with the above, we get
\begin{equation}
I_\triang(\hat E_i [\Zar^{S_0}]) = q I_\triang([\Zar^{S_0}]),
\end{equation}
which concludes part (a).

For part (b), since $\omega(\Delta_j,\ast)=0$, we have
\begin{equation}
\begin{split}
I_\triang([\qZ_j\qZ'_j\qZ''_j][\Zar^{S_0}])
&= I_\triang([\Zar^{S_0+\Delta_i}]) \\
&=\sum_{[S] \in Q_0(\triang;\BZ)/\BD}
(-q^{\frac{1}{2}})^{-\chi(S)} q^{\frac{1}{2} \omega(S_0,S)} J(-S_0-\Delta_i + S) \,.
\end{split}
\end{equation}
By \eqref{eq-J-diff}, the $-\Delta_i$ can be replaced by a factor of
$-q^{\frac{1}{2}}$, which proves part (b). Finally, part (c) follows similarly from
the linear $q$-difference equation~\eqref{eq-J-Lagr} for the tetrahedron index.

This completes the proof of the theorem.
\end{proof}

\begin{remark}
Still assume that $M$ has no non-peripheral $\ints/2$-homology. There is a bijection
of $Q_0(\triang;\BZ)/\BD$ with the set of embedded generalized normal surfaces
(see~\cite[Eqn.(35)]{GHHR}) which converts the sum~\eqref{3dIdef2} as a generating
function over the set of embedded generalized normal surfaces (i.e., embedded
surfaces that intersect each tetrahedron in a polygon with a number of sides
divisible by $4$). 
\end{remark}

\subsection{Compatibility with 3--2 and 2--0 moves}

In this section we study how the map~\eqref{It2} changes under moves on the
triangulation. 
By the same proof as parts (b) and (c) of the invariance in Theorem~\ref{thm.3d},
the map
\begin{equation}
I_N:\bigotimes_{j=1}^N\qtorusev\langle\qZ_j,\qZ_j',\qZ''_j\rangle\to\indexR,\qquad
I_N([\Zar^S])=J(-S)
\end{equation}
descends to $\bigotimes_{j=1}^N\qgluev(\lantern_j)\to\indexR$. Note the homology
condition is only used to show that there are no extra relations from edges, so the
tetrahedron solutions and Lagrangians are not affected.

Given $S=S(k)\in\BE$, note
\begin{equation}
\Emon_S:=(-q^{\frac{1}{2}})^{-\chi(S)}[\Zar^{-S}]=\prod_{i=1}^{N-r}(q\qE_i^{-1})^{k_i},
\end{equation}
and $q^{\frac{1}{2}\omega(S_0,S)}$ is the Weyl-ordering factor for
$\big[[\Zar^{-S}][\Zar^{S_0}]\big]$. Then we have yet another formula for the index
\begin{equation}
\label{eq-Itriang-def}
I_\triang(x)=\sum_{S\in\BE}I_N(\Emon_S\cdot x).
\end{equation}

\begin{proposition}
\label{prop-I-moves}
Suppose $M$ has no non-peripheral $\ints/2$-homology.
The triangulation index is compatible with the maps~\eqref{phi32} and~\eqref{phi20} 
\begin{equation}
\label{Iphi}
I_{\triang_2}=I_{\triang_3}\circ\phi_{3,2}, \qquad
I_{\triang_0}=I_{\triang_2}\circ\phi_{2,0} \,.
\end{equation}
\end{proposition}

\begin{proof}
This is a straightforward generalization of Theorem~4.3 and A.1 in \cite{GHRS}, so we
omit some details.

First consider the 3--2 move. We use the notations of Section~\ref{sec-qtr-32}. As in
\cite[Theorem~A.1]{GHRS}, we can choose the independent edges of $\triang_2$ such
that adding the new edge $E_0$ of $\triang_3$ gives an independent set of edges in
$\triang_3$. Thus,
\begin{equation}
I_{\triang_3}(\phi_{3,2}(x))=
\sum_{S\in\BE_2}\sum_{k\in\ints}
I_{N+1}\left((q\qE_0^{-1})^k\cdot\phi_{3,2}(\Emon_S\cdot x)\right).
\end{equation}

Let $x$ be a monomial and write $\Emon_S\cdot
x=[\qZ^a\qZ'^b\qZ''^c][\qW^d\qW'^e\qW''^f]\tilde{x}\in\qgluev(\triang_2)$. By
definition, $\phi_{3,2}(\Emon_S\cdot x)$ is the monomial
$[\qX'^{a+f}\qX''^{b+d}][\qY'^{b+e}\qY''^{c+f}][\qV'^{c+d}\qV''^{a+e}]\tilde{x}$.
Then the sum over $k$ can be evaluated using the pentagon identity
\eqref{eq-J-penta}, recalling that $\qE_0=\qX\qY\qV$. The remaining sum is exactly
$I_{\triang_2}(x)$.

The 2--0 move is similar. We use the notations of Section~\ref{sec-qtr-20}. Recall
the shared edge in the move region of $\triang_0$ is denoted $E$, which splits into
two edges $E',E''$ in $\triang_2$, and the additional central edge in $\triang_2$ is
denoted $E_0$.

Using \cite[Theorem~4.3]{GHRS}, we can choose $E$ to be part of the independent edges
of $\triang_0$, and a set of independent edges for $\triang_2$ can be obtained by
replacing $E$ with $E_0,E',E''$. Write
\begin{equation}
\BE_0=\ints E\oplus\tilde{\BE}_0,\qquad
\BE_2=\ints E_0\oplus\ints E'\oplus \ints E''\oplus\tilde{\BE}_0.
\end{equation}
Then
\begin{align}
I_{\triang_0}(x)&=
\sum_{S\in\tilde{\BE}_0}\sum_{a\in\ints}I_N\left((q\qE^{-1})^a\cdot \Emon_S\cdot x\right),\\
I_{\triang_2}(\phi_{2,0}(x))&=
\sum_{S\in\tilde{\BE}_0}\sum_{(a,b)\in\ints^2}\sum_{k\in\ints}
I_{N+2}\left((q\qE_0^{-1})^k(q\qE'^{-1})^a(q\qE''^{-1})^b\cdot
\phi_{2,0}(\Emon_S\cdot x)\right).
\end{align}
In $I_{\triang_2}(\phi_{2,0}(x))$, the sum over $k$ can be evaluated using
\eqref{eq-J-quad}, resulting in $q^{-a}\delta_{a,b}$, deleting $\qZ$ and $\qW$ from
$\qE'$ and $\qE''$, and combining them into $(q\qE^{-1})^a$. This agrees with
$I_{\triang_0}(x)$.
\end{proof}

Combining this with Propositions~\ref{prop-qtr-32} and \ref{prop-qtr-20}, we get
the commutative diagrams of Equation~\eqref{movediag}.

\begin{remark}\label{rem-extra-sum}	
Now that all the pieces are defined, we sketch the generalization of the 3d-index to
manifolds with non-peripheral $\ints/2$-homology.

In this case, by \cite[Theorem~7.1]{GHHR}, $Q_0(\triang;\ints)$ is only a subgroup of
the lattice of normal surfaces $N(\triang;\ints)$, and the quotient is isomorphic to
the cokernel of $H_1(\partial M;\ints/2)\to H_1(M;\ints/2)$, which is finite and
independent of the triangulation $\triang$.

The first instances of the homology restriction are Lemma~\ref{lemma-basis-change}
and Proposition~\ref{prop-qgluev}. The lattice $\Lambda$ there is isomorphic to
$Q_0(\triang;\ints)/\BD$, and we can show that $\Lambda^\perp$ is isomorphic to
$N(\triang;\ints)/\BD$, and $\Lambda^\perp=(\frac{1}{2}\Lambda)\cap\ints^{2N}$. The
even edge relation $\qE_i-q=[\Zar^{E_i}]-(-q^{1/2})^{-\chi(E_i)}$ should be
generalized to $[\Zar^{S}]-(-q^{1/2})^{-\chi(S)}$ for all $S\in N(\triang;\ints)/\BD$
when there is non-peripheral $\ints/2$-homology.

Proposition~\ref{prop-qtr-20} also has the homology restriction, but it follows from
the presentation in Proposition~\ref{prop-qgluev}. As mentioned in
Remark~\ref{rem-20}, 2--0 move can be modified to work in the full $\qglue(\triang)$,
so the homology restriction is naturally dropped in this case.

Instead of summing over $Q_0(\triang;\ints)/\BD$ in \eqref{3dIdef2}, the
generalization of the index should sum over $N(\triang;\ints)/\BD$. By partitioning
into cosets of $N(\triang;\ints)/Q_0(\triang;\ints)$, we see that the generalization
is an extra finite sum, so the convergence and compatibility with moves are
unaffected. Since the presentation of $\qgluev(\triang)$ has the same modification,
the descent to $I_M$ also works.
\end{remark}

\subsection{Insertion of peripheral elements}
\label{sec-peri-index}

Recall the vectors $M_k,L_k\in\ints^{3N}$ defined in Section~\ref{sub.Idef}, which
are associated to a basis $\mu_k,\lambda_k$ of $H_1(\partial M;\ints)$ . We get a
vector $C_\gamma$ for each $\gamma\in H_1(\partial M;\ints)$ by making linear
combinations. Define
\begin{equation}
\qm_\gamma=[\vec{z}^{C_\gamma}]\in
\bigotimes_{j=1}^N\qtorus\langle\qz_j,\qz_j',\qz''_j\rangle.%\onto\qglue(\triang).
\end{equation}
By the symplectic relations among $M_k,L_k$, the monomials $\qm_\gamma$ from
different boundary components commute with each other. Therefore, we focus on a
single component in the following.

If $\gamma$ reduces to 0 in $H_1(M;\ints/2)$, then $\qm_\gamma$ is in the even part.
Then the DGG index with class $\gamma$ is simply $I_\triang(\qm_\gamma)$. In terms of
charges, we have Equation~\eqref{I2DGG}. If we choose $\lambda$ to be the homological
longitude. Then $m$ can take half integer values.

Although $\qm_\gamma$ is not technically in $\skein(M)$, it can be included by an
extension. This requires a quick review of the skein algebra of the torus $T^2$.

Model the torus as $T^2=\reals^2/\ints^2$ such that the curves in the directions
$(1,0)$ and $(0,1)$ are the longitude $\lambda$ and meridian $\mu$. Given a nonzero
$\gamma\in H_1(T^2;\ints)$, we can write $\gamma=s\lambda+t\mu$. If $s,t$ are
coprime, define $K_\gamma$ as the simple closed curve in the direction $(s,t)$. More
generally, if $\gcd(s,t)=d$, then $K_\gamma\in\skein(T^2)$ is defined as
$T_d(K_{s/d,t/d})$, where $T_n(x)$ is the Chebyshev polynomial given by
\begin{equation}
T_0(x)=2,\qquad T_1(x)=x,\qquad T_n(x)=xT_{n-1}(x)-T_{n-2}(x),\quad n\ge 2.
\end{equation}
These elements together with $1$ form a basis of $\skein(T^2)$.
%A key property of the Chebyshev polynomial is $T_n(t+t^{-1})=t^n+t^{-n}$.
This is used in \cite{FG:torus} to obtain an algebra embedding
\begin{equation}\label{eq-torus-embed}
\skein(T^2)\embed\qtorus\langle \qm_\lambda,\qm_\mu\rangle,\qquad
K_\gamma\mapsto\qm_\gamma+\qm_\gamma^{-1}.
%[\qm_\lambda^s \qm_\mu^t]+[\qm_\lambda^s \qm_\mu^t]^{-1}.
\end{equation}
% Here, $\qm_\mu,\qm_\lambda$ are simply generators of a quantum torus with
% $q$-commuting relations $\qm_\lambda \qm_\mu=q^{1/2}\qm_\mu \qm_\lambda$.
% Of course, we can choose them to be the monomials defined above.
% Note $[\qm_\lambda^s\qm_\mu^t]=\qm_\gamma$.

Now back to the 3-manifold $M$. For each torus boundary component, a choice of the
longitude $\lambda$ and meridian $\mu$ defines a map $\skein(T^2)\to\skein(M)$. In
fact, $\skein(T^2)$ acts on $\skein(M)$ since gluing $T^2\times[-1,1]$ to the torus
boundary does not change the manifold. We choose to glue along $T^2\times\{-1\}$ so
that the action is on the left.

A related action is defined on the quantum gluing module. By the symplectic
properties mentioned in Section~\ref{sub.Idef}, the peripheral monomials $\qm_\gamma$
commutes with the edges $\qe_i$. Thus, the quantum torus $\qtorus\langle
\qm_\lambda,\qm_\mu\rangle$ acts on $\qglue(\triang)$ from the left.

\begin{theorem}
\label{thm-peri}	
The actions above are related by the quantum trace and the embedding
\eqref{eq-torus-embed}. In other words,
%$\qtr_\triang(K_\gamma)=\qm_\gamma+\qm_\gamma^{-1}$. More generally,
\begin{equation}
\qtr_\triang(K_\gamma\cdot\alpha)=(\qm_\gamma+\qm_\gamma^{-1})\cdot\qtr_\triang(\alpha).
\end{equation}
\end{theorem}

The classical trace obviously has this property, so by the classical limit given in
\cite[Theorem~1.1]{GY}, the quantum trace is given by the same formula up to
$q^{1/8}-q^{-1/8}$. The nontrivial part of the theorem is proving that all quantum
corrections vanishes, which is not obvious since certain intermediate stages could
have correction terms (see e.g.\ Lemma~\ref{lemma-htex}). The proof 
is given in Appendix~\ref{sec-peri-proof}. 

The theorem implies that the composed index $I_M:\skeinev(M)\to\indexR$
from~\eqref{Icomp} can be extended to
\begin{equation}
\label{eq-peri-ext}
\qtorus\langle \qm_\mu^2,\qm_\lambda\rangle\otimes_{\skein'(T^2)}\skeinev(M)
\to\indexR.
\end{equation}
Here, $\skein'(T^2)$ is the subalgebra of $\skein(T^2)$ where the homology grading is
even in $\mu$ but arbitrary in $\lambda$, which is chosen to be a homological
longitude. This corresponds to the index $\mathcal{I}_{M+\qvar{O}_K}(m,e)$
conjectured by \cite{AGLR}, where $m,e$ defines the monomial in
$\qtorus\langle\qm_\mu^2,\qm_\lambda\rangle$ in the same way as the DGG index, and
$K\in\skeinev(M)$.

\subsection{Chirality}

In this section we discuss the behavior of our maps under the (involution) operation
of reversing the orientation of the ambient 3-manifold. As always, $M$ denotes a
compact oriented 3-manifold with torus boundary and $\triang$ an ideal triangulation
of $M$. Denote by $-M$ the same manifold with the opposite orientation.

Starting with the skein module, we see the crossing in \eqref{eq-skein} is flipped
when the orientation is reversed. Recalling the involution $\iota$ from~\eqref{iota},
there is an $\iota$-conjugate linear map
$\skein(M)\to\skein(-M)$ sending a framed link to itself, denoted also by $\iota$ by
abuse of notation. For the stated skein algebras and corner-reduced skein modules, to
preserve the defining relations \eqref{eq-arcs}, \eqref{eq-stex}, and
\eqref{eq-crdef}, it is also necessary to flip the states. Then we also get an
$\iota$-conjugate algebra map $\iota:\skein(\surface)\to\skein(-\surface)$ and an
$\iota$-conjugate linear map $\iota:\skeincr(\surface)\to\skeincr(-\surface)$.

Now consider the quantum gluing modules. $-M$ has a triangulation $-\triang$ obtained
from $\triang$ by relabeling the vertices to change the orientations of all
tetrahedra. This can be done systematically by exchanging vertices 1 and 3 for all
tetrahedra, which has the effect $\qz\leftrightarrow\qz''$. Considering the state
flipping in skein modules, the correct correspondence is
\begin{equation}
\iota(\qz)=\qz^{\prime\prime-1},\quad \iota(\qz'')=\qz^{-1}.
\end{equation}
The inverse is also necessary for the Lagrangian equation. Putting it all together,
we get an $\iota$-conjugate map $\iota:\qglue(\triang)\to\qglue(-\triang)$.

These maps can be combined into the following commutative diagram.
\begin{equation}
\begin{tikzcd}
\skeinev(M) \arrow[r,"\qtr_\triang"] \arrow[d,"\iota"] &
\qgluev(\triang) \arrow[r,"I_\triang"] \arrow[d,"\iota"] &
\indexR \arrow[d,"\iota"] \\
\skeinev(-M) \arrow[r,"\qtr_{-\triang}"] &
\qgluev(-\triang) \arrow[r,"I_{-\triang}"] & \indexR
\end{tikzcd}
\end{equation}
The left square commutes by definition and the discussion above. The right square
commutes by definition and the duality \eqref{eq-index-duality}. This means
$I_{-M}\circ\iota=\iota\circ I_M$.

\cite{AGLR} conjectured a symmetry
\begin{equation}
I_M(\qm_\gamma\otimes K)=I_M(\qm_\gamma^{-1}\otimes K)
\end{equation}
for the peripheral-extended 3d-index \eqref{eq-peri-ext}. We did not find any
theoretical reason for this to hold in general. In the specific example of $4_1$ that
\cite{AGLR} considered, it can be explained by the reversibility of $4_1$. On a
related note, the restriction to DGG index, that is, when $K=1$, satisfies
\begin{equation}
I_M(\qm_\gamma^{-1})=\iota I_M(\qm_\gamma)=I_M(\qm_\gamma)(q^{-1}).
\end{equation}
This is obtained by replacing $k\leftrightarrow-k$ in the sum in \eqref{3dIdef}.

\section{Example: The $4_1$ knot}
\label{sec.compute}

Let $M$ be the complement of the $4_1$ knot. The default triangulation $\triang$ of
$M$ with isometry signature \texttt{cPcbbbiht\_BaCB} has two tetrahedra $T_0$ and
$T_1$ shown in Figure~\ref{fig-41-triang} using the convention from
Figure~\ref{fig-smoothL}. There are 4 face pairings labeled A,B,C,D.

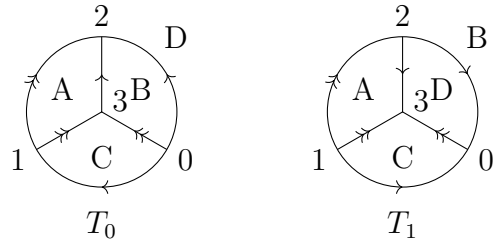
\begin{figure}[htpb!]
\centering
\begin{tikzpicture}
\tikzmath{\r=1;\s=0.6;}
\draw (0,0) circle(\r) node[anchor=-150]{3};
\draw (0,0) -- (-150:\r)node[anchor=30]{1} (0,0) -- (-30:\r)
node[anchor=150]{0} (0,0) -- (90:\r)node[above]{2};
\path[-o-={0.5}{>}] (0,0) -- (90:\r);
\path[-o-={0.5}{>}] (-30:\r) arc[radius=\r,start angle=-30,delta angle=120];
\path[-o-={0.5}{>}] (-30:\r) arc[radius=\r,start angle=-30,delta angle=-120];
\path[-o-={0.5}{>>}] (-30:\r) -- (0,0);
\path[-o-={0.5}{>>}] (-150:\r) -- (0,0);
\path[-o-={0.5}{>>}] (-150:\r) arc[radius=\r,start angle=-150,delta angle=-120];
\path (0,-\r-0.5)node{$T_0$} (30:\s)node{B} (150:\s)node{A} (-90:\s)
node{C} (45:\r+0.4)node{D};
\begin{scope}[xshift={\r*4cm}]
\draw (0,0) circle(\r) node[anchor=-150]{3};
\draw (0,0) -- (-150:\r)node[anchor=30]{1} (0,0) -- (-30:\r)
node[anchor=150]{0} (0,0) -- (90:\r)node[above]{2};
\path[-o-={0.5}{>}] (90:\r) -- (0,0);
\path[-o-={0.5}{>}] (90:\r) arc[radius=\r,start angle=90,delta angle=-120];
\path[-o-={0.5}{>}] (-150:\r) arc[radius=\r,start angle=-150,delta angle=120];
\path[-o-={0.5}{>>}] (-30:\r) -- (0,0);
\path[-o-={0.5}{>>}] (-150:\r) -- (0,0);
\path[-o-={0.5}{>>}] (-150:\r) arc[radius=\r,start angle=-150,delta angle=-120];
\path (0,-\r-0.5)node{$T_1$} (30:\s)node{D} (150:\s)node{A} (-90:\s)node{C} (45:\r+0.4)
node{B};
\end{scope}
\end{tikzpicture}
\caption{\texttt{SnapPy} triangulation of the $4_1$ knot.}
\label{fig-41-triang}
\end{figure}

The dual surface $\surface_\triang$ is obtained in \cite[Section~6.2]{GY} and
reproduced in Figure~\ref{fig-41-lantern}, which is $\surface_\triang$ split along
the $A$-circles. The region labeled $i+4j$ in black corresponds to the vertex $i$ in
$T_j$, and dually, the red $A$-circle labeled $i+4j$ is the face opposite to vertex
$i$ in $T_j$. The small red labels around the $A$-circles indicate how the lanterns
glue together to form $\surface_\triang$. The arrow on a blue standard arc indicates
which edge of $\triang$ it goes around.

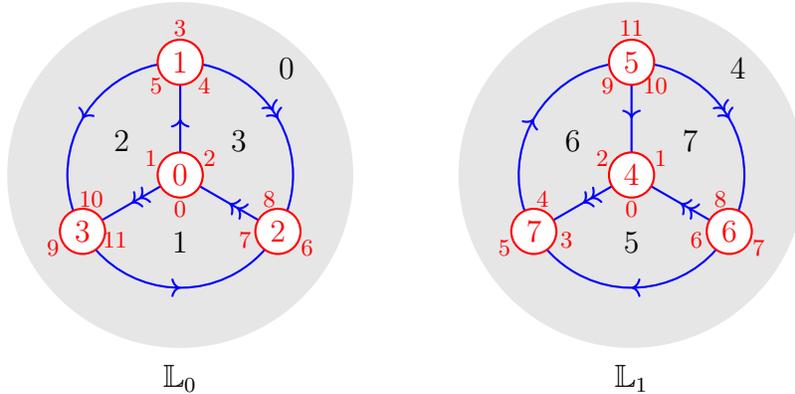
\begin{figure}[htpb!]
\centering
\begin{tikzpicture}[baseline=(ref.base)]
\tikzmath{\r=1.5;\s=0.3;\d=0.45;\e=\r*0.6;}
\fill[gray!20] (0,0)circle(\r+0.8);
\begin{scope}[thick,blue,radius=\r]
\draw[-o-={0.5}{>}] (90:\r) arc[start angle=90,delta angle=120];
\draw[-o-={0.5}{>}] (-150:\r) arc[start angle=-150,delta angle=120];
\draw[-o-={0.5}{>>}] (90:\r) arc[start angle=90,delta angle=-120];
\draw[-o-={0.5}{>}] (0,0) -- (90:\r);
\draw[-o-={0.5}{>>}] (0,0) -- (-150:\r);
\draw[-o-={0.5}{>>}] (-30:\r) -- (0,0);
\end{scope}
\draw[thick,red,fill=white,radius=\s] (90:\r)circle node{1} (-30:\r)
circle node{2} (-150:\r)circle node{3} (0,0)circle node{0};
\path[red,font=\scriptsize] (90:\r) +(90:\d)node{3} +(-45:\d)node{4} +(-135:\d)node{5}
(-30:\r) +(-30:\d)node{6} +(-165:\d)node{7} +(105:\d)node{8}
(-150:\r) +(-150:\d)node{9} +(75:\d)node{10} +(-15:\d)node{11}
(0,0) +(-90:\d)node{0} +(150:\d)node{1} +(30:\d)node{2};
\path (30:\e)node{3} (150:\e)node{2} (-90:\e)node{1} (45:\r+0.5)node{0};
\path (-90:\r+1.2)node{$\lantern_0$};
\begin{scope}[xshift=\r*4cm]
\fill[gray!20] (0,0)circle(\r+0.8);
\begin{scope}[thick,blue,radius=\r]
\draw[-o-={0.5}{>}] (-30:\r) arc[start angle=-30,delta angle=-120];
\draw[-o-={0.5}{>}] (-150:\r) arc[start angle=-150,delta angle=-120];
\draw[-o-={0.5}{>>}] (90:\r) arc[start angle=90,delta angle=-120];
\draw[-o-={0.5}{>}] (90:\r) -- (0,0);
\draw[-o-={0.5}{>>}] (0,0) -- (-150:\r);
\draw[-o-={0.5}{>>}] (-30:\r) -- (0,0);
\end{scope}
\draw[thick,red,fill=white,radius=\s] (90:\r)circle node{5} (-30:\r)
circle node{6} (-150:\r)circle node{7} (0,0)circle node{4};
\path[red,font=\scriptsize] (90:\r) +(90:\d)node{11} +(-45:\d)node{10} +(-135:\d)
node{9} (-30:\r) +(-30:\d)node{7} +(-165:\d)node{6} +(105:\d)node{8}
(-150:\r) +(-150:\d)node{5} +(75:\d)node{4} +(-15:\d)node{3}
(0,0) +(-90:\d)node{0} +(150:\d)node{2} +(30:\d)node{1};
\path (30:\e)node{7} (150:\e)node{6} (-90:\e)node{5} (45:\r+0.5)node{4};
\path (-90:\r+1.2)node{$\lantern_1$};
\end{scope}
\node (ref) at (0,0){\phantom{$-$}};
\end{tikzpicture}
\caption{Lantern diagram for the $4_1$ knot.}\label{fig-41-lantern}
\end{figure}

The quantum trace of the meridian $\mu$ was calculated in \cite[Section~6.2]{GY} in
agreement with Theorem~\ref{thm-peri}. Let $K_b$ be the knot in $M$ shown in
Figure~\ref{fig-Kb}, and let $K_b^{(2)}$ be the 2-parallel. By
\cite[Theorem~2.1]{BL:twist}, $\mu^n\cdot K_b$ and $\mu^n\cdot K_b^{(2)}$ is a basis
of $\skein(M)$. Thus, $\qtr_\triang(K_b)$ and $\qtr_\triang(K_b^{(2)})$ determine the
entire quantum trace map by Theorem~\ref{thm-peri}.
%Kb is the unknot component in L8n2

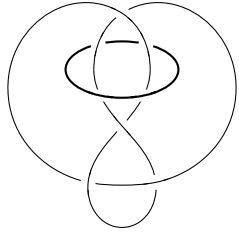
\begin{figure}[htpb!]
\begin{tikzpicture}[scale=1.5]
\node (A) at (0,1) {};
\node (B) at (0,0) {};
\node (L) at (-0.3,-0.5) {};
\node (R) at (0.3,-0.5) {};
\begin{knot}[background color=white]
  \strand (0.25,0.5) to[out angle=90,curve through={(A)..(-1,0.2)..(L)},in angle=180]
  (0,-0.525);
  \strand (0,-0.525) to[out angle=0,curve through={(R)..(1,0.2)..(A)},in angle=90]
  (-0.25,0.5);
\strand (-0.25,0.5) to[out angle=-90,curve through={(B)..(R)},in angle=0] (0,-0.9);
\strand (0,-0.9) to[out angle=180,curve through={(L)..(B)},in angle=-90] (0.25,0.5);
\strand[thick,x radius=0.5,y radius=0.25] (0.5,0.5) arc[start angle=0,delta angle=360];
\flipcrossings{2,7,8}
\end{knot}
\end{tikzpicture}
\caption{The knot $K_b$ in the complement of $4_1$.}
\label{fig-Kb}
\end{figure}

The lantern diagram of $K_b$ is given in Figure~\ref{fig-Kb-lantern}. To calculate
its quantum trace, we assign matching states to the endpoints of the arcs. Because of
the standard arc in $\lantern_0$, the only nonzero terms are the all $+$ and all $-$
states. Then by twisting the arc in $\lantern_1$ using Lemma~\ref{lemma-tw}, we get
\begin{equation}
\label{eq-Kbtr}
\qtr_\triang(K_b)=q^{-1/2}(\qz'^{-1}_0(\qz'^{-1}_1-\qz'_1)-\qz'_0\qz'^{-1}_1)
\in \qglue(\triang) \,.
\end{equation}
Since $K_b$ has 0 linking number with $4_1$, it has 0 homology in $M$, and it is an
element of the even skein module $\skeinev(M)$. However, the diagram
Figure~\ref{fig-Kb-lantern} does not have 0 homology on $\surface_\triang$, so the
quantum trace \eqref{eq-Kbtr} calculated from this diagram is not obviously in the
even gluing module $\qgluev(\triang)$. This can be fixed by using
$\qe_0=[\qz_0^2\qz'_0\qz_1^2\qz'_1]$ and $[\qz_j\qz'_j\qz''_j]=\iunit q^{1/4}$ to
rewrite $\qtr_\triang(K_b)=(-q^{1/2}\qe_0^{-1})\qtr_\triang(K_b)$ as
\begin{equation}
\qtr_\triang(K_b)%=(-q^{1/2}\qe_0^{-1})\qtr_\triang(K_b)
=-q^{-1/2}(\qZ_0^{-1}\qZ''_1+\qZ''_0\qZ_1^{-1}+\qZ''_0\qZ''_1) \in \qgluev(\triang) \,.
\end{equation}
% This shows that $\qtr_\triang(K_b)$ is in fact in $\qgluev(\triang)$.
This agrees with the conjecture
of \cite{AGLR} and matches the calculations of \cite{PP} up to conventions. The same
calculation can be done with $K_b^{(2)}$, the result given by
\begin{equation}
  \qtr_\triang(K_b^{(2)})
  =q^{-2}\qZ'^{-1}_0\qZ'^{-1}_1+q^{-1}(\qZ'_0\qZ'^{-1}_1+\qZ'^{-1}_0\qZ'_1)\\
-(q^{-1}+q^{-2})(\qZ'^{-1}_0+\qZ'^{-1}_1)+q^{-2}+1.
\end{equation}
%\begin{equation}
%\begin{split}
%  \qtr_\triang(K_b^{(2)})&
%  =q^{-2}\qZ'^{-1}_0\qZ'^{-1}_1+q^{-1}(\qZ'_0\qZ'^{-1}_1+\qZ'^{-1}_0\qZ'_1)\\
%&\quad-(q^{-1}+q^{-2})(\qZ'^{-1}_0+\qZ'^{-1}_1)+q^{-2}+1.
%\end{split}
% \end{equation}
This time, the diagram already has even homology, so the result is manifestly in
$\qgluev(\triang)$.
This also matches \cite{AGLR,PP}. The authors in~\cite{PP} claimed that their
calculation does not match with \cite{AGLR}, but their results are actually equal in
$\qglue(\triang)$.

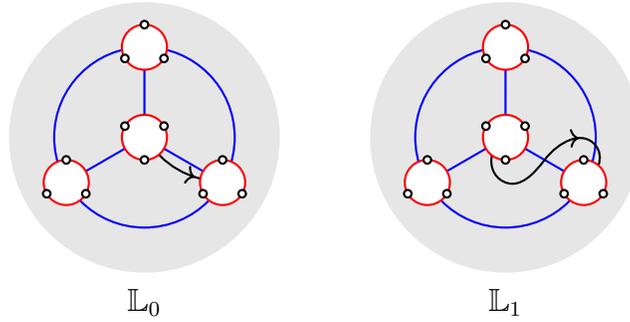
\begin{figure}[htpb!]
\centering
\begin{tikzpicture}[baseline=(ref.base),thick]
\tikzmath{\r=1.2;\s=0.3;\d=0.45;\e=\r*0.6;}
\fill[gray!20] (0,0)circle(\r+0.6);
\draw[blue] (0,0)circle(\r) foreach \t in {90,-30,-150} {(0,0) -- (\t:\r)};
\draw[-o-={0.7}{>}] (0,0) to[out=-55,in=175] (-30:\r);
\draw[red,fill=white] (90:\r)circle(\s) (-30:\r)circle(\s) (-150:\r)
circle(\s) (0,0) circle(\s);
\foreach \t in {-30,-150,90}
\draw[fill=white] (-\t:\s)circle(0.05) (\t:\r) foreach \q in {0,120,-120}
{+(\t+\q:\s)circle(0.05)};
\path (-90:\r+1)node{$\lantern_0$};
\begin{scope}[xshift=\r*4cm]
\fill[gray!20] (0,0)circle(\r+0.6);
\draw[blue] (0,0)circle(\r) foreach \t in {90,-30,-150} {(0,0) -- (\t:\r)};
\draw[-o-={0.7}{>}] (0,0)
to[out angle=-150,curve through={(-90:\r/2)..(-30:\r/2)..(-5:\r)},in angle=30] (-30:\r);
\draw[red,fill=white] (90:\r)circle(\s) (-30:\r)circle(\s) (-150:\r)
circle(\s) (0,0)circle(\s); \path (-90:\r+1)node{$\lantern_1$};
\foreach \t in {-30,-150,90}
\draw[fill=white] (-\t:\s)circle(0.05) (\t:\r)
foreach \q in {0,120,-120} {+(\t+\q:\s)circle(0.05)};
\end{scope}
\node (ref) at (0,0){\phantom{$-$}};
\end{tikzpicture}
\caption{Lantern diagram for $K_b$.}
\label{fig-Kb-lantern}
\end{figure}

Finally, 3d-index of $K_b$ is given by
\begin{equation}
  I_M(K_b)=-q^{-1/2}\sum_{k\in\ints}2q^kJ_\Delta(2k,k,-1)J_\Delta(2k+1,k,0)
  +q^{2k}J_\Delta(2k, k,-1)^2.
\end{equation}
Using methods described in \cite[Section~7]{GHRS}, we can find the first few terms
in the series expansion.
\begin{equation}
\begin{aligned}
-q^{1/2}I_M(K_b)
&=-3q-q^2+7q^3+15q^4+22q^5+11q^6-11q^7-60q^8+O(q^9),\\
\iota(-q^{1/2}I_M(K_b))
&=1-3q-6q^2-q^3+9q^4+28q^5+39q^6+45q^7+20q^8+O(q^9).
\end{aligned}
\end{equation}
%% see Mathematica file: mathematica/quantum.trace.map/index-with-insersion-v2.nb
The second line matches the calculations in \cite{AGLR}.

\subsection*{Acknowledgements}

The authors wish to thank Tudor Dimofte, Thang L\^e, and Mauricio Romo for
enlightening conversations.

%%%%%%%%%%%%%%%%%%%%%%%%%%%%%%%%%%%%%%%%%%%%%%%%%%%%%%%%%%%%%%%%%%%%%%%%%%%%
%%%%%%%%%%%%%%%%%%%%%%%%%%%%%%%%%%%%%%%%%%%%%%%%%%%%%%%%%%%%%%%%%%%%%%%%%%%%

\appendix

\section{Proof of Theorem~\ref{thm-peri}}
\label{sec-peri-proof}

We only need to prove the theorem for a set of $\gamma$ for which $K_\gamma$
generates $\skein(T^2)$. By \cite{BP:mult}, the skein algebra $\skein(T^2)$ is
generated by 3 elements $K_m,K_l,K_{m+l}$ for some basis $m,l$ of $H_1(T^2;\ints)$.
We start with the choice of $m,l$ to make the diagrams as simple as possible, and
then argue that quantum corrections vanish for these choices.

\subsection{Setting up the diagram}

%First, we put the curve $K_\gamma$ into a convenient position. 

The triangulation $\triang$ induces a triangulation $\triang_\partial$ of the
boundary torus. The fundamental domain of $\triang_\partial$ is (topologically) a
parallelogram. Let $m,l$ be the sides of the parallelogram with some orientation.
They define a basis of $H_1(T^2;\ints)$, and $m+l$ is a diagonal path of the
fundamental domain. Let $\gamma$ be any one of $m,l,m+l$. Then by pushing the path
slightly to the right, we obtain a simple closed curve $K_\gamma$ in general position
with the 1-skeleton of $\triang_\partial$. Moreover, there is at most one arc in each
triangle of $\triang_\partial$, and the arc is one of two types: a ``small"
counterclockwise arc or a ``big'' clockwise arc. See Figure~\ref{fig-substd-curve}.

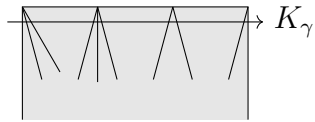
\begin{figure}[htpb!]
\centering
\begin{tikzpicture}
\draw[fill=gray!20] (0,-1.5) |- (3,0) -- (3,-1.5);
\draw (0,0) edge +(-75:1) edge +(-60:1);
\draw (1,0) edge +(-105:1) edge +(-90:1) edge +(-75:1);
\draw (2,0) edge +(-105:1) edge +(-75:1);
\draw (3,0) edge +(-105:1);
\draw[->] (-0.2,-0.2) -- (3.2,-0.2) node[right]{$K_\gamma$};
\end{tikzpicture}
\caption{The curve $K_\gamma$.}
\label{fig-substd-curve}
\end{figure}

The diagram $D_\gamma\subset\surface_\triang$ for $K_\gamma$ on the dual surface can
be obtained by the process in \cite[Section~6.1]{GY}, which we briefly recall here.
$\surface_\triang$ is obtained from $\triang_\partial$ by truncating the vertices and
rearranging the truncated triangles into lanterns. Note the curve $K_\gamma$ avoids
the vertices of $\triang_\partial$. Then the process turns $K_\gamma$ into
$D_\gamma$. The punctures on the boundary triangles of lanterns correspond to the
midpoints of edges of $\triang_\partial$.

In this process, small arcs become standard arcs, and big arcs become arcs that are
parallel to standard arcs but one click away at both endpoints. We call the second
type \term{substandard}. This is shown in Figure~\ref{fig-substd-arc}, where the blue
dashed arc is the truncation at the vertex of $\triang_\partial$, the blue solid arc
is standard as usual, and the black arc is substandard.

\begin{figure}[htpb!]
\centering
\begin{tikzpicture}[baseline=(ref.base)]
\draw[fill=gray!20] (0,0) -- (2,0) -- (60:2) -- cycle;
\draw[fill=white] (1,0)circle(0.05) (60:1)circle(0.05) +(1,0)circle(0.05);
\draw[blue,densely dashed] (0.4,0) arc[radius=0.4,start angle=0,end angle=60] (1.6,0)
arc[radius=0.4,start angle=180,end angle=120] (60:1.6)
arc[radius=0.4,start angle=-120,end angle=-60];
\draw[blue,thick,-o-={0.65}{>}] (0.6,0) arc[radius=0.6,start angle=0,end angle=60];
\draw[thick,-o-={0.4}{>}] (0.8,0) arc[radius=1.2,start angle=180,end angle=120];
\node (ref) at (1,{sqrt(3)/2}) {\phantom{$-$}};
\end{tikzpicture}
$\Longrightarrow$\quad
\begin{tikzpicture}[baseline=(ref.base)]
\tikzmath{\r=1.2;\s=0.3;\d=0.45;\e=\r*0.6;}
\fill[gray!20] (0,0)circle(\r+0.6);
\draw[blue,dashed] (0,0)circle(\r) foreach \t in {-30,-150,90} {(0,0) -- (\t:\r)};
\draw[thick,blue,-o-={0.6}{>}] (-150:\r) to[out=5,in=-125] (0,0);
\draw[thick,-o-={0.6}{>}] (-150:\r) to[out=-5,in=-175] (-30:\r);
\draw[thick,red,fill=white] (90:\r)circle(\s) (-30:\r)
circle(\s) (-150:\r)circle(\s) (0,0)circle(\s);
\foreach \t in {-30,-150,90}
\draw[fill=white] (-\t:\s)circle(0.05) (\t:\r)
foreach \q in {0,120,-120} {+(\t+\q:\s)circle(0.05)};
\node (ref) at (0,0) {\phantom{$-$}};
\end{tikzpicture}
\caption{Standard and substandard arcs.}
\label{fig-substd-arc}
\end{figure}
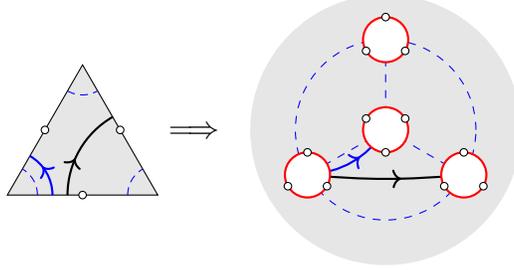

To calculate the quantum trace, we need to apply the splitting homomorphism to
$D_\gamma$, and the choice of height orders for splitting is crucial. A convenient
choice is to start at an intersection with an $A$-circle and assign monotonically
decreasing height to $D_\gamma$. Right before coming back to the starting point, a
sharp increase of height is needed to close up.

Since $K_\gamma$ passes through each triangle of $\triang_\partial$ at most once, the
diagram $D_\gamma$ has at most one arc in each region of the lanterns bounded by
standard arcs. Then each boundary edge of the lanterns has at most 2 endpoints.

If there is a boundary edge with exactly one endpoint, then increasing height in the
discussion above is actually irrelevant, since only the height order is preserved by
isotopy. Then we can write
\begin{equation}
\label{eq-peri-cut-I}
  \cut_A(D_\gamma)=\sum_{s_1,\dotsc,s_k=\pm}D^1_{\gamma,s_1s_2}\cup D^2_{\gamma,s_2s_3}
  \cup\dotsm\cup D^k_{\gamma,s_ks_1},
\end{equation}
where $D^i_{\gamma,\mu\nu}$ is the $i$-th arc of $D_\gamma$ with states $\mu,\nu$ at
the starting and ending points, respectively. We call this type of $\gamma$ type I.
Here, $\cup$-product is used since the splitting \eqref{eq-split} does not have the
correction factor from \eqref{eq-new-prod}.

If there is no such boundary edge, then we are forced to deal with the increasing
segment. We use the trick of coaction, where we double the $A$-circle and put the
segment in an annulus, which is later absorbed by the counit. By listing the finitely
many combinations of standard and substandard arcs, we see that the two segments in
the annulus always have opposite orientations. Therefore, the local picture looks
like Figure~\ref{fig-edge-2end}. Then the result of splitting is
\begin{equation}
\label{eq-peri-cut-II}
\cut_A(D_\gamma)=\sum_{s_1,\dotsc,s_k=\pm}
\epsilon\left(
\begin{linkdiag}
\fill[gray!20] (0,0)rectangle(1,1);
\draw[->] (0,0)--(0,1); \draw[<-] (1,0)--(1,1);
\draw[thick] (0,0.3) -- (1,0.3) (0,0.7) -- (1,0.7);
\draw (0,0.7)\stnodel{s_{t+1}} (0,0.3)\stnodel{s_k} (1,0.7)
\stnode{s_t} (1,0.3)\stnode{s_1};
\end{linkdiag}
\right)
D^1_{\gamma,s_1s_2}\cup\dotsm\cup D^{t-1}_{\gamma,s_{t-1}s_t}\cup
D^{t+1}_{\gamma,s_{t+1}s_{t+2}}\cup\dotsm\cup D^{k-1}_{\gamma,s_{k-1}s_k}.
\end{equation}
As before, $D^i_\ast$ denotes the $i$-th segment. The $t$-th and $k$-th segments are
in the counit. We call this type of $\gamma$ type II.

\begin{figure}[htpb!]
\centering
\begin{tikzpicture}[baseline=(ref.base)]
\fill[gray!20] (0,0) rectangle (1.5,2);
\draw (0,0) -- (0,2) (1.5,0) -- (1.5,2);
\foreach \x in {0,1.5} \foreach \y in {0.2,1.8}
\draw[fill=white] (\x,\y)circle(0.05);
\path[-o-={0.95}{>}] (1.5,1.8) -- (1.5,0.2);
\path[-o-={0.95}{>}] (0,0.2) -- (0,1.8);
\draw[thick] (0,0.5) -- +(1.5,0) (0,1.5) -- +(1.5,0);
\draw[thick,dashed,-o-={0.5}{>}] (0,1.5) arc[radius=0.5,start angle=90,end angle=270];
\draw[thick,dashed,-o-={0.5}{>}] (1.5,0.5) arc[radius=0.5,start angle=-90,end angle=90];
\draw[-Stealth] (2,0.2)node[right,inner sep=2pt]{\small start} -- (1.55,0.45);
\node (ref) at (1,1) {\phantom{$-$}};
\end{tikzpicture}
\caption{Local picture of 2 intersections with an edge.}
\label{fig-edge-2end}
\end{figure}
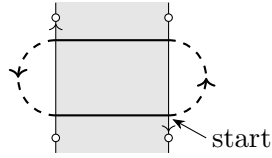

\begin{comment}
Note in \eqref{eq-peri-cut-I} and \eqref{eq-peri-cut-II}, by definition, the product
should be $\cup$, but the additional factor from \eqref{eq-new-prod} cancels for
paired boundary triangles, so the equations hold for $\cdot$ product as well. See
e.g.\ the end of the proof of \cite[Theorem~5.3]{GY}.
\end{comment}

\subsection{Height exchanges and $q$-commuting relations}

The quantum trace $\qtr_\triang(K_\gamma)$ is obtained from the splitting by applying
relations such as those in Lemma~\ref{lemma-tw} to reduce everything to standard
arcs. However, the fact that these relations only hold at the bottom (as a
consequence of the left ideal) means we cannot reduce each arc independently in an
obvious way. This issue is even more apparent in the calculation of
$\qtr_\triang(K_\gamma\cdot\alpha)$.

A careful maneuver is required to bring the arc we want to reduce to the right end of
the product. The moves below follow from \cite[Lemma~2.4]{Le:decomp}, but we use form
given in \cite[Equation~(50)]{LS}.

\begin{lemma}\label{lemma-htex}
In $\skein(\surface)$, we have the following height exchange moves.
\begin{align}
\label{eq-htex}
\relacross{\mu}{\nu}&=q^{\mu\nu/4}\relacross[<-]{\mu}{\nu}
+\delta_{\mu<\nu}q^{-1/4}(q^{1/2}-q^{-1/2})\relacross[<-]{\nu}{\mu}.\\
\label{eq-htex-neg}
\relacross[<-]{\mu}{\nu}&=q^{-\mu\nu/4}\relacross{\mu}{\nu}
+\delta_{\mu<\nu}q^{1/4}(q^{-1/2}-q^{1/2})\relacross{\nu}{\mu}.
\end{align}
\end{lemma}

The following corollary is checked by direct calculation. For convenience, we say a
stated arc is \term{diagonal} if the states at both endpoints are the same.

\begin{corollary}
\label{cor-qcomm}	
The following $q$-commutation relations holds in the skein algebra
$\skein(\lantern)$.
\begin{enumerate}
\item Suppose $a,b$ are disjoint arcs with the same state at all 4 endpoints, then
  they $q$-commute.
\item Let $a$ be an off-diagonal arc with endpoints on different boundary edges. Then
$a$ $q$-commutes with arcs isotopic to it but with possibly different states.
\item Suppose $a$ and $b$ are segments of $D_\gamma$ in the same lantern, and $a$ is
standard. If $a$ has $+$ and $-$ states at the starting and ending points
respectively, then $a$ and $b$ $q$-commute.
\end{enumerate}
\end{corollary}

\begin{proof}
(1) This directly follows from the lemma since $\mu=\nu$ implies that the second term
vanishes for both moves.

(2) This follows from the observations in \cite[Section~4.6]{GY}.

\begin{figure}[htpb!]
\centering
\begin{tikzpicture}[baseline=(ref.base)]
\tikzmath{\r=1.2;\s=0.3;\d=0.45;\e=\r*0.6;}
\fill[gray!20] (0,0)circle(\r+0.6);
\draw[thick,blue,-o-={0.5}{>}] (0,0) to[out=60,in=-60] (90:\r);
\draw[thick,-o-={0.3}{>}] (90:\r)..controls(150:\s+0.2)..(-150:\r);
\draw[thick,-o-={0.4}{>}] (-150:\r)..controls +(105:1) and +(120:1)..(0,0);
\draw[blue,dashed] (0,0)circle(\r) foreach \t in {-30,-150,90} {(0,0) -- (\t:\r)};
\draw[thick,red,fill=white] (90:\r)circle(\s) (-30:\r)
circle(\s) (-150:\r)circle(\s) (0,0)circle(\s);
\path[blue,inner sep=2pt] (60:\s)node[anchor=60]{$+$} (90:\r) +(-60:\s)
node[anchor=-60]{$-$};
\foreach \t in {-30,-150,90}
\draw[fill=white] (-\t:\s)circle(0.05) (\t:\r)
foreach \q in {0,135,-135} {+(\t+\q:\s)circle(0.05)};
\node (ref) at (0,0) {\phantom{$-$}};
\end{tikzpicture}
\caption{$q$-commuting positions for the standard arc.}
\label{fig-qcomm-std}
\end{figure}
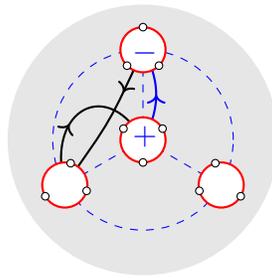

(3) If $a$ and $b$ do not have endpoints on the same boundary edge, then this follows
from definition \eqref{eq-new-prod}. If $b$ is isotopic to $a$ but with possibly
different states, then this follows from (2). The only remaining cases are shown in
Figure~\ref{fig-qcomm-std}. Then the $q$-commuting property follows from direct
calculations using the previous lemma.
\end{proof}

Let $c$ be a substandard arc on $\lantern$ with its ``big clockwise'' orientation.
Define $c_{\mu\nu}$ as $c$ with states $\mu$ and $\nu$ at the starting and ending
points respectively. See Figure~\ref{fig-substd-local}. By Lemma~\ref{lemma-tw},
$c_{+-}$ (by itself) reduces to an off-diagonal standard arc in $\skeincr(\lantern)$
of the form in Corollary~\ref{cor-qcomm}(3).

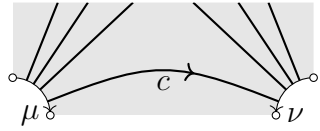
\begin{figure}[htpb!]
\centering
\begin{tikzpicture}
\tikzmath{\h=2;\w=2*\h;\s=0.5;}
\begin{scope}
\clip (90:\s) arc[radius=\s,start angle=90,end angle=0] -- (\w-\s,0)
	arc[radius=\s,start angle=180,end angle=90] |- (0,1.5) -- cycle;
\fill[gray!20] (0,0) rectangle (\w,1.5);
\draw[thick,-o-={0.6}{>}] (0,0) ..controls (\h,0.8).. (\w,0)
node[midway,below,inner sep=2pt]{$c$};
\draw[thick] (0,0) edge (0.3*\h,1.5) edge (0.5*\h,1.5) edge (0.8*\h,1.5)
(\w,0) edge (1.2*\h,1.5) edge (1.5*\h,1.5) edge (1.7*\h,1.5);
\end{scope}
\draw[-o-={0.95}{>}] (90:\s) arc[radius=\s,start angle=90,end angle=0];
\draw[-o-={0.95}{>}] (\w,\s) arc[radius=\s,start angle=90,end angle=180];
\draw[fill=white] foreach \t in {0,90} {(\t:\s)circle(0.05) (\w,0)+(180-\t:\s)
  circle(0.05)};
\path (\s,0)node[left]{$\mu$} (\w-\s,0)node[right]{$\nu$};
\end{tikzpicture}
\caption{Locally picture of $c_{\mu\nu}\cup\alpha$.}
\label{fig-substd-local}
\end{figure}

For convenience, we call $\alpha\in\skeincr(\lantern)$ \term{standard off-diagonal}
if it is a sum of products of standard arcs where each term has an off-diagonal arc.

\begin{lemma}\label{lemma-substd}
Let $\alpha$ be a diagram on $\lantern$ disjoint from the substandard arc $c$ up to
isotopy.
\begin{enumerate}
\item
  If $(\mu,\nu)=(+,-)$, then $c_{+-}$ and $\alpha$ $q$-commute in $\skein(\surface)$.
\item
  If $\mu=\nu$, and $\alpha$ is a product of standard arcs, then in
  $\skeincr(\lantern)$, $(c_{\mu\mu}-s_{-\mu})\cdot\alpha$ is standard off-diagonal.
  Here, $s_{-\mu}$ is the diagonal standard arc $s$ parallel to $c$ with state $-\mu$
  at both endpoints.
\end{enumerate}
\end{lemma}

\begin{proof}
The local picture of $c_{\mu\nu}\cup\alpha$ is given in
Figure~\ref{fig-substd-local}, and recall $c_{\mu\nu}\cdot\alpha$ differs from it by
the factor in \eqref{eq-new-prod}. For both parts, we need to apply
Lemma~\ref{lemma-htex} repeatedly to bring certain endpoints to the lowest for corner
reductions.

First look at (1). This is a more advanced version of Corollary~\ref{cor-qcomm}(3)
and essentially the same proof as \cite[Theorem~7.1]{CL}. On the $\mu=+$ side, we
apply \eqref{eq-htex}. Then the second term is 0 every time, so the entire height
exchange is a single term with a power of $q$ depending on which states are passed
through. The $\nu=-$ side is similar. The result of the height exchanges is
$q^k\alpha\cup c_{+-}$ for some $k$. Therefore, $c_{+-}$ and $\alpha$ $q$-commute
(for both $\cup$ and $\cdot$).

Next consider (2) with $\mu=\nu=+$. In anticipation of the application of
Lemma~\ref{lemma-tw}, let $d'_L$ and $d_R$ be the gradings of $\alpha$ on the left
and right edges in Figure~\ref{fig-substd-local}, respectively. This also defines
$d_L,d''_L,d'_R,d''_R$ by the convention of Lemma~\ref{lemma-tw}.

The $\mu=+$ side is the same as (1) where the total power of $q$ is $q^{d'_L/4}$. The
$\nu=+$ side now has corrections from the second term of \eqref{eq-htex-neg}, but the
leading term still safely move the $+$ state to the bottom with a factor of
$q^{-d_R/4}$. In a correction term, $\nu=+$ passes through a $-$ state and the states
get switched. This means the substandard arc becomes $c_{+-}$, so by part (1), it can
be moved to the bottom, which then becomes standard off-diagonal after twisting, as
mentioned before the lemma. Therefore,
\begin{equation}
c_{++}\cup\alpha=q^{\frac{1}{4}(d'_L-d_R)}\alpha\cup c_{++}+\alpha',
\end{equation}
where $\alpha'$ is standard off-diagonal. Now we can twist $c_{++}$ using
Lemma~\ref{lemma-tw} to get
\begin{equation}
\label{eq-substd-postw}
\alpha\cup c_{++}=(\iunit q^{\frac{1}{8}(d_L-(d'_L+1)+3)})(
-\iunit q^{\frac{1}{8}((d_R+1)-d'_R-3)})\alpha\cup s_-+\text{(off-diagonal term)}.
\end{equation}
The off-diagonal term comes from the $+$ state term in the second part of
\eqref{eq-twdown}, which can be combined with $\alpha'$.
% This term can be twisted in the opposite direction, resulting in $c_{+-}$.
% (See also \cite[Equation~(34)]{GY}.) Thus, the additional term

Next, we try to flip the product order in $\alpha\cup s_-$. Standard arcs away from
$s$ $\cup$-commutes with $s_-$, and the rest of the standard arcs $\cup$-commutes
with $s_-$ up to standard off-diagonal terms. See e.g.\ \cite[Section~4.6]{GY}. Thus,
$\alpha\cup s_--s_-\cup\alpha$ is off-diagonal, and we can flip the $\cup$-product
order in \eqref{eq-substd-postw}.

Finally, we can put in the factors from \eqref{eq-new-prod} for both
$c_{++}\cup\alpha$ and $s_-\cup\alpha$. After everything is combined, we get
\begin{equation}
c_{++}\cdot\alpha=q^{\frac{1}{4}(d_L-d'_R)}s_-\cdot\alpha+\alpha'.
\end{equation}
If $\alpha$ is already off-diagonal, then this shows $c_{++}\cdot\alpha$ is standard
off-diagonal, and so is $(c_{\mu\mu}-s_{-\mu})\cdot\alpha$. If $\alpha$ is indeed
diagonal, then $d_L=d'_R$, so $(c_{\mu\mu}-s_{-\mu})\cdot\alpha=\alpha'$ is standard
off-diagonal. This proves part (2) for $\mu=\nu=+$. The case of $\mu=\nu=-$ is
similar, so we omit the calculation here.
\end{proof}

\subsection{Type I}

We start with the simpler case of type I $\gamma$.

Given $\alpha\in\skein(M)$, the quantum trace $\qtr_\triang(\alpha)$ is represented
by $\sum_i r_i\alpha_i$ where $r_i\in\ground$ and
$\alpha_i\in\bigotimes_j\skeincr(\lantern_j)$ is a product of standard arcs. Since
the reduction from $\alpha$ to $\alpha_i$ happens ``below'' $D_\gamma$, we also have
\begin{equation}
\label{eq-peri-act-cut}
\qtr_\triang(K_\gamma\cdot\alpha)=\sum_{s_1,\dotsc,s_k=\pm}
q^{w_s}D^1_{\gamma,s_1s_2}\cdot D^2_{\gamma,s_2s_3}
\dotsm D^k_{\gamma,s_ks_1}\cdot\sum_ir_i\alpha_i,
\end{equation}
where $w_s$ is the correction factor for turning $\cup$ into $\cdot$. Suppose not all
states $s_\ast$ are the same, then there exists some $(s_j,s_{j+1})=(+,-)$, where
$s_{k+1}=s_1$ for convenience. If $D^j_{\gamma,s_js_{j+1}}$ is standard, then by
Corollary~\ref{cor-qcomm}(3), it can be brought to the front of the product, so the
term becomes $0$ in $\qglue(\triang)$. If it is substandard instead, then we use
Lemma~\ref{lemma-substd} to bring it to the bottom. As noted before
Lemma~\ref{lemma-substd}, by twisting it into a standard arc,
Corollary~\ref{cor-qcomm}(3) applies again to show that the term is $0$ in
$\qglue(\triang)$.

This shows that there are only two potentially nonzero choices of states in
\eqref{eq-peri-act-cut}. Take the all $+$ state term. We can twist the substandard
segments of $D_\gamma$ using Lemma~\ref{lemma-substd} from the lowest to the highest.
If $\alpha_i$ is off-diagonal, then it is a $0$ term in $\qtr_\triang(\alpha)$ by
definition, and it is also a $0$ term in \eqref{eq-peri-act-cut} by
Lemma~\ref{lemma-substd}. If $\alpha_i$ is diagonal, then up to terms that will
eventually become $0$ in $\qglue(\triang)$, each substandard segment can be replaced
by the corresponding standard arc with $-$ states by Lemma~\ref{lemma-substd}. The
same process is applied to the all $-$ state term, and we get
\begin{equation}
\label{eq-peri-act-tw}
\qtr_\triang(K_\gamma\cdot\alpha)=\sum_{s=\pm}q^{w_s}
\tilde{D}^1_{\gamma,s}\cdot\tilde{D}^2_{\gamma,s}\dotsm\tilde{D}^k_{\gamma,s}
\cdot\qtr_\triang(\alpha),
\end{equation}
where $\tilde{D}^t_{\gamma,s}$ is the standard arc $D^t_{\gamma,ss}$ if the $t$-th
segment of $K_\gamma$ is counterclockwise, and it is the standard arc parallel to
$D^t_{\gamma,\ast}$ with state $-s$ if the $t$-th segment of $K_\gamma$ is clockwise.
% Here, unlike \eqref{eq-peri-cut-I}, we must use $\cdot$-product as required
% by Lemma~\ref{lemma-substd}.

The Weyl-ordering of the product of $\tilde{D}$ in \eqref{eq-peri-act-tw} is the
definition of $\qm_\gamma^{\pm1}$, so to prove the theorem, the last step is to show
that the Weyl-ordering factor cancels $q^{w_s}$. This follows from the following
claims.
\begin{enumerate}
\item
  The original arcs $D^1_{\gamma,ss},D^2_{\gamma,ss},\dotsc,D^k_{\gamma,ss}$ $q$-commute.
\item
  The $q$-commuting relations of $\tilde{D}^\ast_{\gamma,s}$ are the same as
  $D^\ast_{\gamma,ss}$.
\item
  $D^1_{\gamma,ss}\cup D^2_{\gamma,ss}\cup \dotsm\cup D^k_{\gamma,ss}$ is the
  Weyl-ordering of $D^1_{\gamma,ss}\cdot D^2_{\gamma,ss}\dotsm D^k_{\gamma,ss}$.
\end{enumerate}

Now we establish the claims. (1) follows from Corollary~\ref{cor-qcomm}(1). (2) is
checked by direct calculation since there are finitely many types of arcs. For (3),
we need to compare the Weyl-ordering $q^{w_s}$ with the factor from
\eqref{eq-new-prod}. Consider a pair $D^i_{\gamma,ss}$ and $D^j_{\gamma,ss}$ where
$i<j$. The $q$-commutation between them has two sources. If they have endpoints on
the same boundary triangle but not the same boundary edge, then the $q$-commutation
comes from \eqref{eq-new-prod}, so this part of Weyl-ordering factor matches. If they
have endpoints on the same boundary edge, then the $q$-commutation comes from height
exchanges. However, since these arcs arise from splitting, there is another pair of
arcs on the other side with the opposite height exchange $q$-commutation. Thus, these
two pairs have cancelling contributions to the Weyl-ordering, and they also have no
contribution in \eqref{eq-new-prod}. This proves (3).

\subsection{Type II}

Now we consider type II $\gamma$. The outline of the proof is the same, so we only
explain the parts that are different.

The formula for the quantum trace is now
\begin{equation}
\begin{split}
\qtr_\triang(K_\gamma\cdot\alpha)=\sum_{s_1,\dotsc,s_k=\pm}
q^{w_s}\epsilon\left(
\begin{linkdiag}
\fill[gray!20] (0,0)rectangle(1,1);
\draw[->] (0,0)--(0,1); \draw[<-] (1,0)--(1,1);
\draw[thick] (0,0.3) -- (1,0.3) (0,0.7) -- (1,0.7);
\draw (0,0.7)\stnodel{s_{t+1}} (0,0.3)\stnodel{s_k} (1,0.7)\stnode{s_t} (1,0.3)
\stnode{s_1};
\end{linkdiag}
\right)
&D^1_{\gamma,s_1s_2}\dotsm D^{t-1}_{\gamma,s_{t-1}s_t}\cdot\\
&\quad D^{t+1}_{\gamma,s_{t+1}s_{t+2}}\dotsm D^{k-1}_{\gamma,s_{k-1}s_k}\cdot\sum_ir_i\alpha_i.
\end{split}
\end{equation}
Again, if not all states are the same, then we can find $(s_j,s_{j+1})=(+,-)$. If
$j\ne t,k$, then the argument is the same as type I. If $j$ is one of $t,k$, then the
counit is zero using height exchanges. See also \cite[Lemma~4.7]{LeYu:SLn}. Thus, the
only nonzero terms are again the ones with all $+$ or all $-$ states. In this case,
the counit is $q^{-1/4}$ for both choices of states.

The twists of the substandard arcs are the same as before, which gives
\begin{equation}
  \qtr_\triang(K_\gamma\cdot\alpha)=\sum_{s=\pm}q^{w_s-1/4}\tilde{D}^1_{\gamma,s}
  \dotsm\tilde{D}^{t-1}_{\gamma,s}\cdot\tilde{D}^{t+1}_{\gamma,s}
  \dotsm\tilde{D}^{k-1}_{\gamma,s}\cdot\qtr_\triang(\alpha).
\end{equation}
Again, we need to argue that the factor $q^{w_s-1/4}$ is exactly the Weyl-ordering.
Claims (1) and (2) are the same as before. Claim (3) is modified to say that the
$\cup$-product times $q^{-1/4}$ is the Weyl-ordering of the $\cdot$-product. The
proof is almost the same except when it comes to the boundary edge that contains the
starting point. \eqref{eq-new-prod} still has no contribution on this edge, but the
$q$-commutations do not cancel on this edge since $s_1$ is higher than $s_t$, while
$s_k$ is lower than $s_{t+1}$. The contributions on this pair of edges is exactly
$q^{-1/4}$. This finishes the proof.

%%%%%%%%%%%%%%%%%%%%%%%%%%%%%%%%%%%%%%%%%%%%%%%%%%%%%%%%%%%%%%%%%%%%%%%%%%%%
%%%%%%%%%%%%%%%%%%%%%%%%%%%%%%%%%%%%%%%%%%%%%%%%%%%%%%%%%%%%%%%%%%%%%%%%%%%%

\bibliographystyle{hamsalpha}
\bibliography{biblio}
\end{document}